\documentclass[12pt]{amsart}
\usepackage{amsmath}
\usepackage{amssymb}
\usepackage{mathdots}
\usepackage{tikz}
\usepackage{adjustbox}
\usepackage[all]{xy}

\numberwithin{equation}{section}

\newtheorem{thm}{Theorem}[section]

\newtheorem{lem}[thm]{Lemma}
\newtheorem{prop}[thm]{Proposition}
\newtheorem{cor}[thm]{Corollary}
\newtheorem{rem}[thm]{Remark}
\newtheorem{exam}[thm]{Example}
\newtheorem{exam-nota}[thm]{Example-Notation}
\newtheorem{rem-nota}[thm]{Remark-Notation}
\newtheorem{nota}[thm]{Notation}
\newtheorem{dfn}[thm]{Definition}

\newtheorem{dfn-nota}[thm]{Definition-Notation}

\newtheorem{dfn-lem}[thm]{Lemma-Definition}
\newtheorem{dfn-prop}[thm]{Proposition-Definition}

\newcommand{\beqa}{\begin{eqnarray*}}
\newcommand{\eeqa}{\end{eqnarray*}}

\newcommand{\fa}{\mbox{${\mathfrak a}$}}

\newcommand{\fk}{\mbox{${\mathfrak k}$}}
\newcommand{\fg}{\mbox{${\mathfrak g}$}}

\newcommand{\fh}{\mbox{${\mathfrak h}$}}

\newcommand{\fp}{\mbox{${\mathfrak p}$}}
\newcommand{\fr}{\mbox{${\mathfrak r}$}}
\newcommand{\fb}{\mbox{${\mathfrak b}$}}

\newcommand{\eps}{\epsilon}

\newcommand{\PR}{\mbox{${\mathbb P}$}}

\newcommand{\C}{\mbox{${\mathbb C}$}}

\newcommand{\Ad}{{\rm Ad}}

\newcommand{\fgl}{\mathfrak{gl}}

\newcommand{\B}{\mathcal{B}}

\newcommand{\ms}{m(s_{\alpha})}

\newcommand{\he}{\hat{e}}

\newcommand{\F}{\mathcal{F}}

\newcommand{\G}{\mathcal{G}}
\newcommand{\Borbitspace}{B_{n-1}\backslash\B_{n}}
\newcommand{\Sh}{\mathcal{S}h}
\newcommand{\calO}{\mathcal{O}}

\newcommand{\Sp}{\mathcal{S}p}
\title[$B_{n-1}$-orbits and the Bruhat graph]{$B_{n-1}$-orbits on the flag variety and the Bruhat graph of the symmetric group}

\author[M. Colarusso]{Mark Colarusso}
\address{Department of Math and Stats, University of South Alabama, Mobile, AL, 36608}
\email{mcolarusso@southalabama.edu}

\author[S. Evens]{Sam Evens}
\address{Department of Mathematics, University of Notre Dame, Notre Dame, IN, 46556}
\email{sevens@nd.edu}


\begin{document}
\maketitle
\begin{abstract}
Let $G=G_{n}=GL(n)$ be the $n\times n$ complex general linear group and embed $G_{n-1}=GL(n-1)$ in the top left hand corner of $G$.  The standard Borel subgroup of upper triangular matrices $B_{n-1}$ of $G_{n-1}$ acts on the flag variety $\B_{n}$ of $G$ with finitely many orbits.  In this paper, we show that each $B_{n-1}$-orbit is the intersection of orbits of two Borel subgroups of $G$ acting on $\B_{n}$.  This allows us to give a new combinatorial description of the $B_{n-1}$-orbits on $\B_{n}$ by associating to each orbit a pair of Weyl group elements.  The closure relations for the $B_{n-1}$-orbits can then be understood in terms of  the Bruhat order on the symmetric group, and the Richardson-Springer monoid action on the orbits can be understood in terms of a well-understood monoid action on the symmetric group.  This approach makes the closure relation more transparent than in \cite{Magyar} and the monoid action significantly more computable than in our papers \cite{CE21I} and \cite{CE21II}, and also allows us to obtain new information about the orbits including a simple formula for the dimension of an orbit.
\end{abstract}

\section{Introduction}

Let $G=G_{n}=GL(n)$ be the $n\times n$ complex general linear group, and let $B_{n-1}$ be the standard Borel subgroup of upper triangular matrices of 
$G_{n-1}$ embedded in the upper left corner of $G$.  In this paper, we give a new combinatorial description of the
$B_{n-1}$-orbits on the flag variety $\B_{n}$ of $G$ by associating to each orbit a pair of elements 
in the Weyl group $W$ of $G$, which  is identified with the symmetric group.   By the Bruhat graph for the symmetric group, we mean the graph whose nodes are given by elements of $W$, and where there is an edge between two nodes if one node corresponds to a Schubert cell that is codimension one in the closure of the Schubert cell corresponding to the other node.  As a consequence, the closure relations studied in \cite{Magyar} and the monoid action studied in our previous papers (\cite{CE21I},\cite{CE21II}) can be studied using the well-known structure of the Bruhat graph.   In particular, they can be deduced from the corresponding Bruhat order relation and monoid actions in the Weyl group.

In more detail, let $B=B_{n}\subset G$ be the standard Borel subgroup of upper triangular matrices stabilizing 
the standard flag $\mathcal{E}_{+}\in\B_{n}$.  Let $B^{*}$ be the Borel subgroup stabilizing the flag $\mathcal{E}^{*}=(V_{1}\subset V_{2}\subset\dots\subset V_{n-1}\subset V_{n})$, where $V_{i}=\mbox{span}\{e_{n},\, e_{1},\dots, e_{i-1}\}$ for $i=1,\dots, n$.  For $Q$ a $B_{n-1}$-orbit, the Bruhat decomposition implies $B\cdot Q=B \cdot w(\mathcal{E}_{+})$ and $B^{*}\cdot Q=B^{*}\cdot u^{*}(\mathcal{E}^{*})$ for unique elements $w, \, u^{*}\in W$.  The first result of this paper proves that the $B_{n-1}$-orbits on $\B_{n}$ are exactly the nonempty intersections of $B$ and $B^{*}$-orbits.  
More precisely,
\begin{thm}\label{thm:introinter}(Theorem \ref{thm:orbitsintersect}, Corollary \ref{c:intersect}, and Corollary \ref{c:Shpair})
The map $\Sh:\Borbitspace \to W\times W$ given by 
\begin{equation}\label{eq:introShmap}
\Sh(Q):=(w,u^{*})\mbox{ where } B\cdot Q=B \cdot w(\mathcal{E}_{+})\mbox{ and } B^{*}\cdot Q =  B^{*}\cdot u^{*}(\mathcal{E}^{*})
\end{equation} 
is injective. Further, if $\Sh(Q)=(w,u^{*})$, then 
$Q=B\cdot w(\mathcal{E}_{+})\cap B^{*}\cdot u^{*}(\mathcal{E}^{*})$, and the image of $\Sh$ consists of all pairs $(w,u^{*})$ such that 
$B\cdot w(\mathcal{E}_{+})\cap B^{*}\cdot u^{*}(\mathcal{E}^{*})$ is nonempty.
\end{thm}

This result was suggested to us by John Shareshian.  We call a pair of Weyl group elements $(w,u^{*})\in W\times W$ with the property that $B\cdot w(\mathcal{E}_{+})\cap B^{*}\cdot u^{*}(\mathcal{E}^{*})\neq \emptyset$ a \emph{Shareshian pair}.  We denote the set of all Shareshian pairs by $\mathcal{S}p\subset W\times W$. Theorem \ref{thm:introinter} implies that the map $\Sh: \Borbitspace\rightarrow \Sp$ in (\ref{eq:introShmap}) is bijective.  We refer to the map $\Sh$ as the \emph{Shareshian map}.   In Proposition \ref{p:Shpairs}, we characterize $\Sp$, and verify that the classification of orbits by Shareshian pairs is equivalent to a classification described by Magyar \cite{Magyar}.  It is also equivalent to an earlier description of orbits given by Hashimoto \cite{Hashi}.


The set of all Shareshian pairs comes equipped with a natural partial order $\leq$ which is the restriction of the product of Bruhat orders on $W\times W$, where in the second factor the Bruhat order is the one defined by the simple reflections 
\begin{equation}\label{eq:introgenerators}
S^{*}=\{s_{1}^{*}, \dots, s_{n-1}^{*}\}\mbox{ with } s_{1}^{*}=(1,n), \mbox{ and } s_{i}^{*}=(i-1,i)\mbox{ for }i=2,\dots, n-1.  
\end{equation}
The second basic result of this paper asserts that 
we can describe the closure relations on $\Borbitspace$ using the Bruhat order on Shareshian pairs.   
\begin{thm}\label{thm:introclosure} (Theorem \ref{t:sharequalbruhat})
Let $Q,\, Q^{\prime}\in\Borbitspace$ and let $\Sh(Q)$ and $\Sh(Q^{\prime})$ denote the corresponding Shareshian Pairs.  Then 
$$
Q^{\prime}\subset\overline{Q}\Leftrightarrow \Sh(Q^{\prime})\leq \Sh(Q). 
$$
\end{thm}
This description of the closure relations seems to us considerably more transparent than the description of the closure relations on $\Borbitspace$ previously given by Magyar in his paper \cite{Magyar}.  See Example \ref{ex:b2flag3} to see the utility of this perspective.   However, our proof relies heavily on the main theorem  from \cite{Magyar}.

 In \cite{CE21I} and \cite{CE21II}, we studied an \emph{extended monoid} action on $\Borbitspace$ by simple roots of $\fk:=\fg_{n-1}$ and $\fg$, which extended a previous action by roots of $\fg$ discussed in Hashimoto's work \cite{Hashi}. 
For a simple root $\alpha$ of $\fk$ or $\fg$ and $B_{n-1}$-orbit $Q$, we let 
$\ms*Q$ denote the monoid action of $\alpha$ on $Q$, and we call the action by roots of $\fk$ the left action and the action by roots of $\fg$ the right action.
  The Weyl group $W$ has well-studied left and right monoid actions by simple roots of $\fg$, whose properties are closely tied to the Bruhat order on $W$.  We use this monoid action on $W$ to define left and right monoid actions by simple roots of $\fk$ and $\fg$ on $W\times W$.
\begin{thm}\label{thm:intromonoid}(Theorem \ref{thm:intertwine})
 The Shareshian map in Equation (\ref{eq:introShmap}) intertwines the extended monoid action on $B_{n-1}$-orbits on $\B_{n}$ with the classical monoid action on $W\times W$ described above, i.e., for a simple root $\alpha$ of $\fk$ or $\fg$ and a $B_{n-1}$-orbit $Q$, 
 \begin{equation}\label{eq:intromonoids}
 \Sh(\ms*Q)=\ms*\Sh(Q).
 \end{equation}
 \end{thm}
\noindent Moreover, each simple root of $\fk$ or $\fg$ has a type, which is essential for studying the geometry of the orbit closures, and we prove in the above Theorem that the type of a root $\alpha$ for $Q$ is determined by the type of the root $\alpha$ for $\Sh(Q)=(w,u^{*})$, which is computed easily using the Bruhat order on $W$.   As a consequence, computation of the extended monoid action becomes far simpler than in our earlier papers \cite{CE21I} and \cite{CE21II}.  
 See Example \ref{ex:b2flag3} and the associated figure and remark.   

The description of the extended monoid action on $\Borbitspace$ given in Theorem \ref{thm:intromonoid} allows us to obtain a simple formula for the dimension of an orbit $Q\in\Borbitspace$. 
\begin{thm}\label{thm:introdim}(Theorem \ref{t:dimformula})
Let $Q\in\Borbitspace$ with $\Sh(Q)=(w, u^{*})$ and let $\sigma\in W$ be the $n$-cycle 
$\sigma=(n,n-1, \dots, 1)$.  Then 
$$
\dim Q=\frac{\ell(w)+\ell(u^{*})+|u^{*}\sigma w^{-1}|-n}{2},
$$
where $|u^{*}\sigma w^{-1}|$ denotes the order of the element $u^{*}\sigma w^{-1}$ in the group $W$, and $\ell(u^{*})$ is the length of $u^{*}$ with respect to the set of simple reflections $S^{*}$ given in (\ref{eq:introgenerators}).  
\end{thm}

We note that we proved in \cite{CE21II} that the closure ordering on $\Borbitspace$ is the so-called standard order of Richardson and Springer discussed in \cite{RS} and this result is used in order to prove the above Theorem.  This paper is part of a larger project.  Indeed, the orbits in $\Borbitspace$ are characterized as the $B$-orbits on $\B_n \times \mathbb{P}^{n-1}$ satisfying a genericity condition for the second factor.  In later work, we plan to give a similar description for all $B$-orbits on $\B_n \times \mathbb{P}^{n-1}$, and to establish further results on the geometry of orbit closures.   In particular, we would like to establish links between our results and the very interesting paper of Travkin \cite{Tr}.  It would also be interesting to understand the connection between this project and the work of Gandini and Pezzini \cite{GP}.

This paper is structured as follows.  In Section \ref{s:orbitintersect}, we prove that each $B_{n-1}$-orbit on $\B_n$ is the intersection of a $B$-orbit and a 
$B^{*}$-orbit, and show that our classification is equivalent to the classification using decorated permutations in \cite{Magyar}.  In Section \ref{s:closure}, we prove that if $Q$ and $Q^{\prime}$ are two orbits in $\Borbitspace$ and $\Sh(Q)=(w,u^{*})$ and $\Sh(Q^{\prime})=(y,v^{*})$, then $Q^{\prime}\subset \overline{Q}$ if and only if $y\le w$ and $v^{*}\le u^{*}$ in the Bruhat order on $W$.  In 
Section \ref{s:monoid}, we recall the definition and basic properties of the monoid action and type of a simple root for an orbit (as studied by many people), define our extended monoid action, and prove the compatibility of the extended monoid action with the classical monoid action on $W\times W$, and explain how to compute the type of an orbit in $\Borbitspace$ from the type of the corresponding $B$ and $B^{*}$-orbits.  We then use this result to compute our formula for the dimension of an orbit $Q$ in $\Borbitspace$.  Finally, we introduce the notion of a standardized Shareshian pair, which makes the order relation slightly more transparent.

We would like to thank John Shareshian who conjectured the result of Theorem \ref{thm:introinter} which inspired this  project.   We would also like to thank the referees for their many useful suggestions, which significantly improve the exposition of this paper.



\section{Parameterizations of $B_{n-1}$-orbits on $\B_{n}$}\label{s:orbitintersect}
In this section, we define the Shareshian map from $B_{n-1}$-orbits on $\B_{n}$ to pairs of Weyl group elements, prove it is injective, and characterize its image, the Shareshian pairs.  We give explicit identifications between three different parameterization of $B_{n-1}$-orbits on $\B_{n}$; flags in standard form, Shareshian pairs, and Magyar's notion of decorated permutations.


\subsection{Conventions and notation }\label{s:notation}
In this paper, all algebraic varieties are by convention complex algebraic varieties, and similarly with Lie algebras.  Let $G=GL(n)$ be the complex general linear group and $\fg=\fgl(n)$ be its Lie algebra.  Embed $G_{n-1}:=GL(n-1)$ in $G$ as matrices fixing the standard basis vector $e_{n}$. We let $\fh\subset \fg$ denote the standard Cartan subalgebra of diagonal matrices and let $H\subset G$ be the corresponding algebraic group.  We let $\eps_{j}\in\fh^{*}$ be the linear functional on $\fh$ which acts on $x= \mbox{diag}[h_{1},\dots, h_{j},\dots, h_{n}]\in\fh$ by $\eps_{j}(x)=h_{j}$.  Let $W=N_{G}(H)/H$ be the Weyl group of $G$ with respect to $H$, which is isomorphic to the symmetric group $\mathcal{S}_n$ on $n$ letters.   If $w \in W$ and $Y$ is a $H$-stable subvariety of a $G$-variety $X$, then $\dot{w}\cdot Y$ is independent of the choice of a representative $\dot{w}$ in $N_G(H)$, and we write $w\cdot Y$ in place of $\dot{w}\cdot Y$.  For an algebraic group $A$ with Lie algebra $\fa$, we denote the adjoint action of $A$ on $\fa$ by $\Ad(g)x$ for $g\in A$ and $x\in\fa$.  Abusing notation, we also denote the action of $A$ on itself by conjugation by $\Ad$, so that $\Ad(g)h:=ghg^{-1}$ for $g,\, h\in A$. 


For a nonzero vector $v\in\C^{n}$, we denote the line through $v$ containing the origin by $[v]\in\mathbb{P}^{n-1}$.  Throughout the paper, we use the identifications of the flag variety $\B_{n}$ with Borel subalgebras of $\fg$ and with the variety
 $\mbox{Flag}(\C^{n})$ of full
flags in $\C^{n}$.  If $\F\in \mbox{Flag}(\C^{n})$ is fixed by $H$ and $w\in W$, we denote the action $w\cdot \F$ of $w$ on $\F$ by $w(\F)$.  We make heavy use of the following notation for flags throughout the paper.  
\begin{nota}\label{nota:standard} Let 
   $$
  \mathcal{F}=(V_{1}\subset V_{2}\subset\dots\subset V_{i}\subset V_{i+1}\subset \dots)
   $$
be a flag in $\C^{n}$, with $\dim V_{i}=i$ and $V_{i}=\mbox{span}\{v_{1},\dots, v_{i}\}$, with each $v_{j}\in\C^{n}$.  We will denote this flag $\mathcal{F}$ 
by
   $$
  \mathcal{F}=  (v_{1}\subset  v_{2}\subset\dots\subset v_{i}\subset v_{i+1}\subset\dots ). 
   $$
\end{nota}


\subsection{Flags in Standard Form} 

In Section 4.1 of \cite{CE21II}, we find a canonical set of representatives for 
elements of $\Borbitspace$ which we call \emph{flags in standard form}.  

\begin{dfn}\label{d:std}
\begin{enumerate}
\item For a standard basis vector $e_{i}\in\C^{n}$ with $i\leq n-1$, we define $\he_{i}:=e_{i}+e_{n}$ and refer to $\he_{i}$ as a hat vector of index $i$.
\item We say that a flag 
\begin{equation}\label{eq:basicflag}
\mathcal{F}:=(v_{1}\subset \dots \subset v_{i}\subset\dots\subset v_{n})
\end{equation}
in the flag variety $\B_n$ for $G$
is in \emph{standard form} if $v_{i}=e_{j_{i}}$ or $v_{i}=\he_{j_{i}}$ for all $i=1,\dots, n$, and 
$\mathcal{F}$ satisfies the following three conditions:
\begin{enumerate}
\item $v_i = e_n$ for some $i.$
\item If $v_{i}=e_{n}$, then $v_{k}=e_{j_{k}}$ for all $k>i$.
 \item If $i<k$ with $v_{i}=\he_{j_{i}}$ and $v_{k}=\he_{j_{k}}$, then $j_{i}>j_{k}$. 
\end{enumerate}
\end{enumerate}
\end{dfn}
\noindent One of the main results of \cite{CE21II} is: 
\begin{thm}\label{thm:standardforms}[see Theorem 4.7 of \cite{CE21II}]
 The map  
\begin{equation*}
\Psi: \{ \text{ Flags in standard form } \} \to B_{n-1}\backslash\B_n, \ \ \mathcal{F}\mapsto B_{n-1}\cdot \F.
\end{equation*}
is bijective.
\end{thm}

\subsection{Shareshian Pairs}\label{ss:Shpairs}
Recall the flags $\mathcal{E}_{+}$ and $\mathcal{E}^{*}$ from the introduction.  
In the language of Notation \ref{nota:standard}, we  write 
$\mathcal{E}_{+}:=(e_{1}\subset e_{2}\subset \dots\subset e_{n})$ and $\mathcal{E}^{*}=(e_{n}\subset e_{1}\subset\dots \subset e_{n-1}).$  Then the standard Borel subgroup $B$ of invertible upper triangular matrices stabilizes the 
flag $\mathcal{E}_{+}$, and we let $B^{*}$ be the stabilizer of $\mathcal{E}^{*}$ in $G$.  
It is easy to see that $B\cap B^{*}=Z B_{n-1}$, where $Z$ is the centre of $G$.

\begin{thm}\label{thm:orbitsintersect}
Let $\F$ be a flag in standard form, and let $B_{n-1}\cdot\F$ be the $B_{n-1}$-orbit through $\F$. Then $B_{n-1}\cdot \F=B\cdot \F\cap B^{*}\cdot \F $.  
\end{thm}

\noindent Consider a $B$-orbit $Q_B$ and a $B^*$-orbit $Q_{B^*}$ in $\B_n$.
If $Q_B \cap Q_{B^*}$ is nonempty, then it contains a $B_{n-1}$-orbit $Q=B_{n-1}\cdot\F$, so by the Theorem, 
\begin{equation}\label{eq:intersection}
Q=B\cdot \F \cap B^*\cdot \F = Q_B \cap Q_{B^*}.
\end{equation} 
Thus, Theorems \ref{thm:standardforms} and \ref{thm:orbitsintersect} imply:  

\begin{cor}\label{c:intersect}
The $B_{n-1}$-orbits on $\B_{n}$ are precisely the nonempty intersections of $B$-orbits and $B^{*}$-orbits.  
\end{cor}

To prove Theorem \ref{thm:orbitsintersect}, we first need to analyze the orbits $B\cdot \F$ and
$B^{*}\cdot \F$.  Since the Borel subgroups $B$ and $B^{*}$ contain the standard Cartan subgroup $H$ of diagonal matrices in $G$, it follows from the Bruhat decomposition that $B\cdot\F$ and $B^{*}\cdot\F$ each contain a unique flag which is $H$-stable.  We compute these flags in the next proposition.  
\begin{nota}\label{n:tildeandstarflags}
Let $\F\in\B_{n}$ be a flag in standard form.  We denote by $\tilde{\F}$ the unique $H$-stable flag in the $B$-orbit $B\cdot \F$ and
by $\F^{*}$ the unique $H$-stable flag in the $B^{*}$-orbit $B^{*}\cdot \F$.
\end{nota}
%
 
\begin{prop}\label{p:orbits}
Let $\F\subset \B_{n}$ be a flag in standard form with 
$$
\F=(v_{1}\subset v_{2}\subset \dots\subset v_{p}\subset \dots\subset v_{n}).
$$  
(1) If $\F$ contains no hat vectors, then 
$\F=\tilde{\F}=\F^{*}$. 

\noindent (2) If $\F$ has hat vectors, we may assume that $\F$ has the form:
\begin{equation}\label{eq:hatvectorF}
\F=(v_{1}\subset\dots\subset v_{i_{k}-1}\subset\underbrace{ \he_{j_{k}}}_{i_{k}}\subset v_{i_{k}+1}\subset \dots \subset \underbrace{\he_{j_{k-1}}}_{i_{k-1}}\subset \dots \subset \underbrace{\he_{j_{1}}}_{i_{1}}\subset \dots\subset \underbrace{e_{n}}_{p}\subset v_{p+1}\subset \dots\subset v_{n}),
\end{equation}
with $j_{k}>j_{k-1}>\dots>j_{1}$ and $v_{m}$ a standard basis vector.  
Then
\begin{equation}\label{eq:tildeF}
\tilde{\F}=(v_{1}\subset\dots\subset v_{i_{k}-1}\subset\underbrace{e_{n}}_{i_{k}}\subset v_{i_{k}+1}\subset \dots\subset \underbrace{e_{j_{k}}}_{i_{k-1}}\subset\dots\subset \underbrace{e_{j_{2}}}_{i_{1}}\subset \dots\subset\underbrace{e_{j_{1}}}_{p}\subset v_{p+1}\subset\dots\subset v_{n}),
\end{equation}
and 
\begin{equation}\label{eq:starF}
\F^{*}=(v_{1}\subset \dots\subset v_{i_{k-1}}\subset \underbrace{e_{j_{k}}}_{i_{k}}\subset v_{i_{k}+1}\subset \dots\subset \underbrace{e_{j_{k-1}}}_{i_{k-1}}\subset \dots\subset \underbrace{e_{j_{1}}}_{i_{1}}\subset \dots\subset\underbrace{e_{n}}_{p}\subset v_{p+1}\subset \dots\subset v_{n}),
\end{equation}
where the $v_{m}$ are the same vectors that appear in the flag in Equation (\ref{eq:hatvectorF}).

\end{prop}
\begin{proof}
If $\F$ contains no hat vectors, then $\F$ is $H$-stable, whence $\F=\tilde{\F}=\F^{*}$. 
Now suppose that $\F$ contains hat vectors.  Equation (\ref{eq:starF}) is clear since 
there exists a $b^{*}\in B^{*}$ so that $b^{*}\cdot \hat{e}_{j_{i}}=e_{j_{i}}$ for $i=1,\dots,k.$ To prove Equation (\ref{eq:tildeF}) consider the element $b\in G$ whose action on the standard basis of $\C^{n}$ is 
given by 
\begin{equation}\label{eq:tostd}
b: e_{n}\mapsto \he_{j_{k}},\; e_{j_{m}}\mapsto -e_{j_{m}}+e_{j_{m-1}} \mbox{ for } m=2,\dots, k,\mbox{ and } b: e_{\ell}\mapsto e_{\ell} \mbox{ for all other } \ell.
\end{equation}
It follows from the definition of the standard form (see Definition \ref{d:std}) that the element $b\in B$.  
We compute 
\begin{equation*}
\begin{split}
b\cdot \tilde{\F}&=(v_{1}\subset\dots\subset \underbrace{\he_{j_{k}}}_{i_{k}}\subset\dots \subset \underbrace{-e_{j_{k}}+e_{j_{k-1}}}_{i_{k-1}}\subset\dots \subset \underbrace{-e_{j_{2}}+e_{j_{1}}}_{i_{1}}\subset \dots\subset \underbrace{e_{j_{1}}}_{p}\subset\dots \subset v_{n})\\
&=(v_{1}\subset\dots\subset \underbrace{\he_{j_{k}}}_{i_{k}}\subset\dots\subset \underbrace{\he_{j_{k-1}}}_{i_{k-1}}\subset\dots\ldots\ldots\subset\underbrace{\he_{j_{1}}}_{i_{1}}\subset\dots\ldots\ldots\subset \underbrace{e_{n}}_{p}\subset\dots\subset v_{n}),
\end{split}
\end{equation*}
which is the flag in Equation (\ref{eq:hatvectorF}). 
\end{proof}

\begin{proof}[Proof of Theorem \ref{thm:orbitsintersect}]
We define a map 
\begin{equation}\label{eq:Phidefn}
\Phi: \{\F\in\B_{n}:\F \mbox{ is a flag in standard form} \}\to \B_{n}\times\B_{n} \mbox{ given by } \Phi(\F)=(\tilde{\F}, \F^{*}),
\end{equation}
where $\tilde{\F}$ and $\F^{*}$ are the flags given in Notation \ref{n:tildeandstarflags}.  
To prove the theorem, note that it suffices to show that the map $\Phi$ is injective.  Indeed, since $B\cap B^{*}=B_{n-1}Z$, for flags $\F$ and $\G$ in standard form, then $B_{n-1}\cdot \F \subset B\cdot\F \cap B^{*}\cdot \F$ and $B_{n-1}\cdot \G \subset B\cdot\G \cap B^{*}\cdot \G.$  If $B_{n-1}\cdot \G$ is another $B_{n-1}$-orbit in $B\cdot \F \cap B^{*}\cdot \F$,  then it follows that $B\cdot\F=B\cdot \G$ and $B^{*}\cdot \F=B^{*}\cdot \G$. Hence, we deduce 
that $\tilde{\F}=\tilde{\G}$ and $\F^{*}=\G^{*}.$  Assuming $\Phi$ is injective, we have
$\F=\G$ and therefore $B_{n-1}\cdot \F=B_{n-1}\cdot \G$. 

To show that the map $\Phi$ in (\ref{eq:Phidefn}) is injective, we need to show that the flag $\F=(v_{1}\subset v_{2}\subset \dots \subset v_{k}\subset \dots \subset v_{n})$ in standard form is uniquely determined by 
the flags $\tilde{\F}$ and $\F^{*}$.  We first note that by Equations (\ref{eq:tildeF}) and (\ref{eq:starF}), the flags $\tilde{\F}$ and $\F^*$ coincide if and only if $\F$ has no hat vectors.   It follows that it suffices to show injectivity of $\Phi$ separately on standard flags without hat vectors and on standard flags with hat vectors. For standard flags $\F$ with no hat vectors then injectivity is clear since $\F=\tilde{\F}=\F^{*}$ by Proposition \ref{p:orbits} (1).    

On the other hand, suppose $\F$ has hat vectors and is written as in Equation \eqref{eq:hatvectorF}.  Let $\tilde{\F}=(\tilde{v}_{1}\subset \tilde{v}_{2}\subset \dots\subset \tilde{v}_{n})$, and let $\F^{*}=(v_{1}^{*}\subset v_{2}^{*}\subset \dots \subset v_{n}^{*})$.  Suppose that $i_{k}$ is the first index such that $\tilde{v}_{i_{k}}\neq v_{i_{k}}^{*}$.  By Equations (\ref{eq:hatvectorF})-(\ref{eq:starF}), $\tilde{v}_{i_{k}}=e_{n}$ and $v_{i_{k}}^{*}=e_{j_{k}}$, and the vector $v_{i_{k}}$ for $\F$ is ${\hat{e}}_{j_{k}}$.  If $\tilde{v}_{i_{k-1}}=e_{j_{k}}$, then $v_{i_{k-1}}={\hat{e}}_{j_{k-1}}$ where $v_{i_{k-1}}^*=e_{j_{k-1}}$.  The remaining hat vectors are determined similarly.  On the other hand, if $\ell$ is such that $\tilde{v}_{\ell}=v_{\ell}^{*}$, then Equations (\ref{eq:hatvectorF})-(\ref{eq:starF}) imply that $v_{\ell}$ is a standard basis vector with $v_{\ell}=\tilde{v}_{\ell}=v^{*}_{\ell}$.  It follows that if $\G$ is another flag in standard form 
with $\tilde{\G}=\tilde{\F}$ and $\G^{*}=\F^{*}$, then $\G=\F$.   Thus, the map 
$\Phi$ in (\ref{eq:Phidefn}) is injective and the proof is complete.  


\end{proof}

\begin{dfn}\label{d:geoShpair}
Let $Q_B$ be a $B$-orbit and let $Q_{B^*}$ be a $B^*$-orbit in $\B_n$.  We call the pair $(Q_B,Q_{B^*})$ a {\it geometric Shareshian pair} if $Q_B \cap Q_{B^*}$ is nonempty, in which case it is a single $B_{n-1}$-orbit by Equation (\ref{eq:intersection}).
\end{dfn}

\begin{rem}\label{r:geoshar}
The map $\Sh$ from $B_{n-1}\backslash \B_n$ to geometric Shareshian pairs given by $Q\mapsto (B\cdot Q, B^*\cdot Q)$ is bijective with the inverse given by taking the intersection of the given $B$ and $B^{*}$-orbits.  This is a restatement of Corollary \ref{c:intersect}.
\end{rem}

We call the map $\Sh$ the {\it Shareshian map}, and we can describe it combinatorially as follows.   By the Bruhat decomposition, we can write any geometric Shareshian pair as $(B\cdot w(\mathcal{E}_{+}),B^*\cdot u^{*}(\mathcal{E}^{*}))$ for unique Weyl group elements $w,\, u^{*}\in W$.
\begin{dfn}\label{d:SHmap}
The Shareshian map is given by:
\begin{equation}\label{eq:Shmap}
\begin{split}
&\Sh: \Borbitspace\to W\times W; \;\Sh(B_{n-1}\cdot\F)=(w,u^{*})\mbox{ where } \F \mbox{ is in standard form and }\\
& \tilde{\F}=w(\mathcal{E}_{+})
\mbox{ and }\F^{*}=u^{*}(\mathcal{E}^{*})\mbox{ are the flags given in Notation \ref{n:tildeandstarflags}.}
\end{split}
\end{equation}
\end{dfn}

\begin{nota}\label{d:Shpair}
 We refer to a pair of Weyl group 
elements $(w,y)\in W\times W$ such that $(w,y)=\Sh(Q)$ for some $Q\in \Borbitspace$,  as a \emph{Shareshian pair}.  We denote the subset of $W\times W$ consisting of all Shareshian pairs as $\Sp\subset W\times W$.  
\end{nota}
\noindent The following statement is immediate from Remark \ref{r:geoshar}.
\begin{cor}\label{c:Shpair}
The Shareshian map $\Sh:B_{n-1}\backslash \B_n \to \Sp$ is bijective.
\end{cor}

\begin{rem}\label{r:starandtildeflags}
Let $Q=\B_{n-1}\cdot \F\in\Borbitspace$ with $\F$ a flag in standard form and with $\Sh(Q)=(w, u^{*})$.  Then it follows from definitions that $B\cdot Q=B\cdot w(\mathcal{E}_{+})$ and $B^{*}\cdot Q=B^{*}\cdot u^{*}(\mathcal{E}^{*})$. Thus, $(w,u^{*})$ is a Shareshian pair if and only if $B\cdot w(\mathcal{E}_{+})\cap B^{*}\cdot u^{*}(\mathcal{E}^{*})\neq\emptyset$.


\end{rem}

We now describe necessary and sufficient conditions for a pair of Weyl group 
elements $(w,y)\in W\times W$ to be a Shareshian pair, after first introducing some notation.  Let $\Delta=\{j_{1}<j_{2}<\dots<j_{k}<n\}$ be a subsequence of $\{1,\dots,n\}$ 
containing $n$ and let $\tau_{\Delta}\in \mathcal{S}_{n}$ be the $k+1$-cycle
$\tau_{\Delta}:=(n, j_{k}, \dots, j_{2}, j_{1})$. Recall the $n$-cycle $\sigma=(n,n-1,\dots, 2,1)\in\mathcal{S}_{n}$ from Theorem \ref{thm:introdim}.  

\begin{prop}\label{p:Shpairs}
 A pair of Weyl group elements 
$(w, y)\in W\times W$ is a Shareshian pair if and only if there exists a subsequence $\Delta = \{j_{1}<j_{2}<\dots<j_{k}<n\}$ of $\{1,\dots, n\}$ containing $n$ such that: 
\begin{enumerate}
\item $y=\tau_{\Delta}w\sigma^{-1}$ and
\item $\Delta$ is a decreasing sequence for $w^{-1}$, i.e., $w^{-1}(n)<w^{-1}(j_{k})<\dots <w^{-1}(j_{1}).$
\end{enumerate}
\end{prop}
\begin{proof}
Suppose first that $Q\in\Borbitspace$.  We claim that $\Sh(Q)=(w,u^{*})$ satisfies conditions (1) and (2) of the
Proposition.  Let $\F\in Q$ be the unique flag in standard form in the orbit $Q$.  First, suppose that 
$\F$ contains no hat vectors, so that $\F=\tilde{\F}=\F^{*}$ by (1) of Proposition \ref{p:orbits}.  Let $\Delta=\{n\}$, so $\tau_{\Delta}=id$.  Since $\sigma^{-1}(\mathcal{E}^{*})=\mathcal{E_{+}}$ and $\tilde{\F}=w(\mathcal{E}_{+})$ and $\F^{*}=u^{*}(\mathcal{E}^{*})$, it follows that $u^{*}=w\sigma^{-1}$, and condition (2) is automatic.

Now suppose that $\F$ contains hat vectors.  Since $\F$ is in standard form, it is  given by (\ref{eq:hatvectorF}), so that $\tilde{\F}$ and $\F^{*}$ are given by Equations (\ref{eq:tildeF}) and (\ref{eq:starF}) 
respectively.  Let $\Delta:=\{j_{1}<j_{2}<\dots<j_{k}<n\}$, so that $\Delta$ consists of the indices of the hat 
vectors in $\F$ along with $n$.  
From Equations (\ref{eq:tildeF}) and (\ref{eq:starF}), we see that $$\F^{*}=u^{*}(\mathcal{E}^{*})=\tau_{\Delta}\tilde{\F}=\tau_{\Delta} w (\mathcal{E}_{+}).$$  
Since $\sigma^{-1}(\mathcal{E}^{*})=\mathcal{E}_{+}$, it follows that $u^{*}=\tau_{\Delta}w\sigma^{-1}$.   From (\ref{eq:tildeF}), we see $w(i_{k})=n,\, w(i_{s})=j_{s+1}$ for $s=1,\dots, k-1$, and $w(p)=j_{1}$ with $i_k < i_{k-1} < \dots < i_1 < p$, so condition (2) follows.


Conversely, suppose that $(w,y)\in W\times W$ satisfies the conditions of (1) and (2) in the statement of the Proposition.  
Let $\Delta=\{j_{1}<\dots< j_{k}<n\}$.  By condition (2), we have $w^{-1}(n)<w^{-1}(j_{k})<\dots< w^{-1}(j_{1})$.  
Let $w^{-1}(n)=i_{k}$, $w^{-1}(j_{k})=i_{k-1}, \dots w^{-1}(j_{2})=i_{1}$, and $w^{-1}(j_{1})=p$.  Let $\mathcal{G}:=w(\mathcal{E}_{+})$.  
Then 
\begin{equation*}
\begin{split}
&\mathcal{G}=
(e_{w(1)}\subset\dots\subset e_{w(i_{k}-1)}\subset\underbrace{e_{n}}_{i_{k}}\subset e_{w(i_{k}+1)}\subset \dots\subset \underbrace{e_{j_{k}}}_{i_{k-1}}\subset\dots\subset \underbrace{e_{j_{2}}}_{i_{1}}\subset \dots\subset\underbrace{e_{j_{1}}}_{p}\subset\\
&\subset e_{w(p+1)}\subset\dots\subset e_{w(n)}),
\end{split}
\end{equation*}
and it follows that 
\begin{equation*}
\begin{split}
&\tau_{\Delta} (\mathcal{G})=(e_{w(1)}\subset \dots\subset e_{w(i_{k-1})}\subset \underbrace{e_{j_{k}}}_{i_{k}}\subset e_{w(i_{k}+1)}\subset \dots\subset \underbrace{e_{j_{k-1}}}_{i_{k-1}}\subset \dots\subset \underbrace{e_{j_{1}}}_{i_{1}}\subset \dots\subset\underbrace{e_{n}}_{p}\subset\\ &\subset e_{w(p+1)} 
\subset \dots\subset e_{w(n)}).
\end{split}
\end{equation*}
Now let $\F$ be the flag in standard form:
\begin{equation*}
\begin{split}
&\F=(e_{w(1)}\subset\dots\subset e_{w(i_{k}-1)}\subset\underbrace{ \he_{j_{k}}}_{i_{k}}\subset e_{w(i_{k}+1)}\subset \dots \subset \underbrace{\he_{j_{k-1}}}_{i_{k-1}}\subset \dots \subset \underbrace{\he_{j_{1}}}_{i_{1}}\subset \dots\subset \underbrace{e_{n}}_{p}\subset\\
&\subset e_{w(p+1)}\subset \dots\subset e_{w(n)}),
\end{split}
\end{equation*}
Then it follows from Proposition \ref{p:orbits} that $\tilde{\F}=\mathcal{G}$ and $\F^{*}=\tau_{\Delta}(\G)$.  
But $\G=w(\mathcal{E}_{+})$ and $$\tau_{\Delta}(\G)=\tau_{\Delta}w(\mathcal{E}_{+})=\tau_{\Delta}w\sigma^{-1}(\mathcal{E}^{*}).$$  
Thus, $\Sh(B_{n-1}\cdot \F)=(w, \tau_{\Delta}w\sigma^{-1})=(w,y)$ by condition (1).  
Thus, $(w,y)\in \Sp$ as asserted. 
\end{proof}

\begin{rem}\label{r:whatisdelta}
Let $\Sh(Q)=(w, u^{*})$ be the Shareshian pair for an orbit $Q=B_{n-1}\cdot \F$ with $\F$ in standard form.  It follows from Proposition \ref{p:Shpairs} and its proof that  
$\tau_{\Delta}:=u^{*}\sigma w^{-1}$ is a cycle of the form $(n, j_{k},\dots, j_{1})$, and the hat vectors 
in $\F$ are $\he_{j_{k}},\dots, \he_{j_{1}}$. We call the decreasing sequence  $\Delta := \{j_{1}<\dots< j_{k}<n\}$ for $w^{-1}$ the \emph{associated decreasing sequence} for  $Q$, and write $\Delta(w,u^*)$ for $\Delta$.
\end{rem}

\begin{exam}
We describe the set of Shareshian pairs $\Sp$ in $\mathcal{S}_{3}\times\mathcal{S}_{3}$ using Proposition \ref{p:Shpairs}.  In this case the cycle $\sigma=(3,2,1)$, so that $\sigma^{-1}=(1,2,3)$.  For $\Delta=\{3\}$, there is no restriction on the permutation $w$ and the cycle $\tau_{\Delta}=id$.  This gives us six such Shareshian pairs of the form $\{(w, w\sigma^{-1}): w\in \mathcal{S}_{3}\}$.  By part (1) of Proposition \ref{p:orbits} and Remark \ref{r:whatisdelta}, these Shareshian pairs correspond to the six $B_{2}$-orbits $Q=B_{2}\cdot w(\mathcal{E}_{+})$ with $w\in\mathcal{S}_{3}$.   Next, suppose $\Delta=\{2,3\}$.  Now there are only three elements $w$ of $\mathcal{S}_{3}$ with the property that $\Delta$ forms a decreasing sequence for $w^{-1}$:  $w=(2,3),\, (1,3),\, (3,2,1)$.  In this case, the cycle $\tau_{\Delta}=(2,3)$ and the corresponding Shareshian pairs are $((2,3), (1,2,3)),\, ((1,3), (3,2,1)), \, ((3,2,1), (2,3)).$  By Remark \ref{r:whatisdelta}, these three Shareshian pairs correspond to $B_{2}$-orbits $Q=B_{2}\cdot \F$ with $\F$ a flag in standard form with the property that the only hat vector that occurs in $\F$ is $\hat{e}_{2}$.    For $\Delta=\{1,3\}$,  there are also three permutations $w$ satisfying the second condition of Proposition \ref{p:Shpairs}.  They are $w=(3,2,1), \, (1,2,3),$ and $(1,3)$.  Now the cycle $\tau_{\Delta}=(1,3)$, and the respective Shareshian pairs are $((3,2,1), (1,3))$, $((1,2,3), (2,3))$, and $((1,3), (1,2,3))$.  These three pairs correspond to $B_{2}$-orbits through flags in standard form $\F$ where $\hat{e}_{1}$ is the only hat vector occurring in $\F$.  Lastly, there is a unique Shareshian pair with $\Delta=\{1,2, 3\}$.  It is $((1,3), (1,3))$ which corresponds to the $B_{2}$-orbit $Q=B_{2}\cdot \F$ where $\F=(\hat{e}_{2}\subset \hat{e}_{1}\subset e_{3})$.   In total there are 13 Shareshian pairs in $\mathcal{S}_{3}\times\mathcal{S}_{3}$ which coincides with the number of $B_{2}$-orbits on $\B_{3}$ given in Example 7.1 of \cite{CE21II} as asserted by Corollary \ref{c:Shpair}. 

\end{exam}

\subsection{The Magyar Parameterization of $\Borbitspace$}\label{ss:Magyar}
In \cite{Magyar}, Magyar studies $GL(n)$-orbits on $\B_{n}\times\B_{n}\times\mathbb{P}^{n-1}$ under the action $g\cdot (\F_{1}, \F_{2}, [v])=(g\cdot \F_{1}, g\cdot \F_{2}, g\cdot [v])$.  By normalizing $\F_{1}=\mathcal{E}_{+}$, we see that these orbits correspond to 
$B$-orbits on $\B_{n}\times\mathbb{P}^{n-1}$.  Let $\mathcal{O}_{n}\subset \mathbb{P}^{n-1}$ be the $B$-orbit through the line $[e_{n}]\in \mathbb{P}^{n-1}$.  Note that $\mbox{Stab}_{B}[e_{n}]=B_{n-1}Z$, where $Z$ is the centre of $GL(n)$.  Thus, $B$-orbits on $\B_{n}\times \calO_{n}$ are in one-to-one correspondence with $B_{n-1}$-orbits on $\B_{n}$ via the correspondence
\begin{equation}\label{eq:Magyarcorres}
B\cdot(\fb^{\prime}, [e_{n}])\longleftrightarrow  B_{n-1}\cdot\fb^{\prime}.
\end{equation}
It is easy to see that this correspondence preserves closure relations. 
The correspondence in (\ref{eq:Magyarcorres}) along with Magyar's parameterization of $B$-orbits on $\B_{n}\times\mathbb{P}^{n-1}$ using his notion of decorated permutations gives us another parameterization of $\Borbitspace$.
\begin{dfn}\label{d:Magyar}
We call a pair $(w, \Delta)$ an $n$-decorated permutation, if $w\in W$ and $\Delta=\{j_{1}<\dots<j_{k}<n\}\subset \{1,\dots, n\}$ is a decreasing sequence for $w^{-1}$.  
\end{dfn}

By Proposition \ref{p:Shpairs}, the map $(w,u^*)\to (w,\Delta(w,u^*))$ (see Remark \ref{r:whatisdelta}) is bijective, giving an identification between Shareshian pairs and $n$-decorated permutations.  Combined with Equation (\ref{eq:Magyarcorres}), this recovers the identification between $B_{n-1}$-orbits on $\B_n$ and $n$-decorated permutations from \cite{Magyar}.

As a consequence, we have given identifications between three different parametrizations of $\Borbitspace$.
\begin{equation}\label{eq:three}
\{ \mbox{Flags in standard form}\} \longleftrightarrow \{\mbox{Shareshian Pairs}\}\longleftrightarrow\{ \mbox{$n$-decorated permutations}\}.  
\end{equation}

\begin{rem}
Hashimoto constructs a parameterization of $\Borbitspace$ in \cite{Hashi}.  It is easy to see that his parameter set $\mathcal{W}$ given in Proposition 2.1 of \emph{loc.cit.} is naturally in bijection with the other parameterizations of $\Borbitspace$ given in (\ref{eq:three}).
\end{rem}


\begin{rem}\label{r:orthocase}
Our work in \cite{CE21I} and \cite{CE21II} also applies in the setting where 
$G=SO(n)$ is the $n\times n$ complex orthogonal group and $G_{n-1}=SO(n-1)$ viewed as a symmetric subgroup of $SO(n)$.  In particular, we prove that orbits of a Borel $B_{n-1}$ of $G_{n-1}$ on the flag variety $\B_{SO(n)}$ are parametrized by certain isotropic flags in standard form
%
%
in Sections 4.2 and 4.3 of \cite{CE21II} and prove the analogue of Theorem \ref{thm:standardforms}
 (see Theorems 4.17 and 4.22 of \cite{CE21II}).  However, we have not found an analogue of a Shareshian pair in the orthogonal setting.  The issue is that using the realizations of $SO(n)$ and $SO(n-1)$ from \cite{CE21II} which have upper triangular Borel subgroups, there is no Borel subgroup $B^{*}\subset SO(n)$ such that $B\cap B^{*}=ZB_{n-1}$, where $B\subset SO(n)$ is the standard Borel subgroup of upper triangular matrices in $SO(n)$ and $Z\subset SO(n)$ is the centre.  
\end{rem}

\section{Closure Relations on $\Borbitspace$ and the Bruhat ordering on Shareshian Pairs}\label{s:closure}

In this section, we prove the second main result of the paper, Theorem \ref{thm:introclosure}, which states that the closure ordering on $\Borbitspace$ is given by the Bruhat ordering on Shareshian pairs by using a related result of Magyar.  
Consider a $B_{n-1}$-orbit $Q=B_{n-1}\cdot \F$ in $\B_n$, where
 $$
  \mathcal{F}=(V_{1}\subset V_{2}\subset\dots\subset V_{n-1} \subset V_n)
   $$
with $\dim(V_i)=i.$
Recall the standard flag $$\mathcal{E}_+=(E_1\subset E_2 \subset \dots \subset E_{n-1} \subset E_n)$$ with $E_{i}=\mbox{span}\{ e_1, e_2, \dots, e_i\}$ for $i=1, \dots, n$ and $$\mathcal{E}^{*}=(E_1^*\subset E_2^* \subset \dots \subset E_{n-1}^* \subset E_n^*)$$ with $E_i^*=\mbox{span}\{ e_n, e_1, e_2, \dots, e_{i-1}\}$ from the Introduction.  Let $L_n = \C e_n.$

We associate to $Q$ the following invariants:

\begin{equation}\label{eq:rpqandstar}
r_{p,q}(Q)=\dim(V_p\cap E_q), \ \ r^*_{p,q}(Q)=\dim(V_p \cap E_q^*), \ \ p, q=1, \dots, n
\end{equation}
and 

\begin{equation}\label{eq:deltarbar}
\delta_{p,q}(Q)=\dim(L_n \cap (V_p + E_q)), \ \ \overline{r}_{p,q}(Q)=r_{p,q}(Q) + \delta_{p,q}(Q), \ \ p,q=0, \dots, n.
\end{equation}

Since $B_{n-1}$ fixes the flags $\mathcal{E}_+$ and $\mathcal{E}^*$ and the line $L_n$, it follows that these invariants depend only on $Q$ and not on the flag 
$\F \in Q.$

\begin{rem}\label{r:connectionwithMag}
Magyar associates invariants $r_{p,q}$ and $\overline{r}_{p,q}$ to a $G$-orbit on 
$\B_{n}\times\B_{n}\times\mathbb{P}^{n-1}$ as follows. For a line $A$ and flags $B_{\bullet}$ and $C_{\bullet}$ with $B_i$ and $C_i$ the $i$-dimensional subspaces of $B_{\bullet}$ and $C_{\bullet}$ respectively, Magyar defines $r_{p,q}(B_{\bullet},C_{\bullet},A)=\dim(B_p \cap C_q)$ and ${\overline{r}}_{p,q}(B_{\bullet},C_{\bullet},A)=r_{p,q}(B_{\bullet},C_{\bullet},A) + \dim(A \cap (B_p + C_q)).$  Applying the equivalences $ \Borbitspace\leftrightarrow B\backslash(\B_{n}\times\mathcal{O}_{n}) \hookrightarrow G\backslash (\B_{n}\times\B_{n}\times\mathbb{P}^{n-1})$ is equivalent to requiring $B_{\bullet} = \mathcal{E}_+$ and $A=L_n.$  It follows that Magyar's invariants equal the invariants $r_{p,q}(Q)$ and $\overline{r}_{p,q}(Q)$ given in Equations (\ref{eq:rpqandstar}) and (\ref{eq:deltarbar}) respectively.  Since the equivalences preserve closure relations, we can apply Magyar's resullts in our setting.
\end{rem}

Let $S^{\prime}=B\cdot \F^{\prime}$ with
$$\mathcal{F}^{\prime}=(V_{1}^{\prime}\subset V_{2}^{\prime}\subset\dots\subset V_{n-1}^{\prime}\subset V_{n}^{\prime}),
   $$
and let $S=B\cdot \F$ with $\F$ as above.  Recall the standard fact
\begin{equation}\label{eq:standardbruhatorder}
S^{\prime} \subset \overline{S} \iff \dim(V_p^{\prime} \cap E_q) \ge \dim(V_p \cap E_q), \,p, q = 1, \dots, n-1
\end{equation}

\noindent from  Proposition 7 in Section 10.5 of \cite{Fulton}.  By changing the standard basis, it follows that if $S^{\prime, *}=B^* \cdot \F^{\prime}$ and $S^*=B^*\cdot \F$

\begin{equation}\label{eq:starstandard}
S^{\prime, *} \subset \overline{S^*} \iff \dim(V_p^{\prime} \cap E_q^*) \ge \dim(V_p \cap E_q^*),\, p, q = 1, \dots, n-1.
\end{equation}


\begin{thm}\label{thm:magyarorder}(\cite{Magyar}, Theorem 2.2)
Let $Q^{\prime}=B_{n-1}\cdot \F^{\prime}$ and let $Q=B_{n-1}\cdot \F$.  Then
$Q^{\prime} \subset \overline{Q}$ if and only if 
\par\noindent (i) $r_{p,q}(Q^{\prime}) \ge r_{p,q}(Q)$ for $p, q = 1, \dots, n-1$, and 
\par\noindent (ii) $\overline{r}_{p,q}(Q^{\prime}) \ge \overline{r}_{p,q}(Q)$ for $p,q = 0, \dots, n-1$.
\end{thm}
\noindent The preceding result is a restatement of a criterion due to Magyar using Remark \ref{r:connectionwithMag}. The next result is used to relate Magyar's criterion to closure relations for $B^{*}$-orbits on $\B_{n}$.

\begin{prop}\label{prop:rstarrbar}
Let $Q$ be a $B_{n-1}$-orbit on $\B_n.$  Then $r^*_{p,q+1}(Q)=\overline{r}_{p,q}(Q)$ for $p,q =0, \dots, n-1.$
\end{prop}

\begin{proof}
Since $E_q^*=\mbox{span}\{ e_n, e_1, \dots, e_{q-1}\},$ it follows that 
\begin{equation}\label{e.starplusline}
E_q^*=E_{q-1} + L_n.
\end{equation}
Let $Q=B_{n-1}\cdot \F$ and let $V_p$ be the $p$-th subspace in the flag $\F$.
   By the second isomorphism theorem for modules,
\begin{equation}
L_n/(L_n \cap (V_p + E_q)) \cong (L_n + V_p + E_q)/(V_p + E_q).
\end{equation}
It follows that
\begin{equation}
  \dim(L_n)-\dim(L_n\cap (V_p + E_q)) = \dim(L_n + V_p + E_q) - \dim(V_p + E_q)
\end{equation}
so that by Equation \eqref{e.starplusline}
\begin{equation}\label{e.deltaleft}
1 - \delta_{p,q}(Q) = \dim(V_p + E_{q+1}^*) - \dim(V_p + E_q).
\end{equation}
Note that 
\begin{equation}\label{e.vpluse}
\dim(V_p + E_q) = p + q - \dim(V_p \cap E_q)=p+q-r_{p,q}(Q)
\end{equation}
and similarly
\begin{equation}\label{e.vplusstar}
\dim(V_p + E_{q+1}^*)=p+q+1-r^*_{p,q+1}(Q).
\end{equation}
By Equations \eqref{e.deltaleft}, \eqref{e.vpluse}, and \eqref{e.vplusstar}, we conclude that 
$$r^*_{p,q+1}(Q)=\delta_{p,q}(Q) + r_{p,q}(Q) = \overline{r}_{p,q}(Q),$$ which verifies the Proposition.
\end{proof}

Recall the map $\Sh:\Borbitspace \to W\times W$ from Equation (\ref{eq:Shmap}) and the Bruhat order on Shareshian pairs given by the restriction of the product of Bruhat orders on $W\times W$, where in the second factor the Bruhat order is the one defined with respect to the simple reflections $S^{*}$ defined in Equation (\ref{eq:introgenerators}) corresponding to the Borel subalgebra $\fb^{*}$.

\begin{thm}\label{t:sharequalbruhat}
Let $Q$ and $Q^{\prime}$ be $B_{n-1}$-orbits on $\B_n$.   Then
 $$
Q^{\prime}\subset\overline{Q}\Leftrightarrow \Sh(Q^{\prime})\leq \Sh(Q). 
$$
\end{thm}

\begin{proof}
By Theorem \ref{thm:magyarorder}, $Q^{\prime} \subset \overline{Q}$ if and only if $r_{p,q}(Q^{\prime}) \ge r_{p,q}(Q)$ for $p,q =1, \dots, n-1$ and 
$\overline{r}_{p,q}(Q^{\prime}) \ge \overline{r}_{p,q}(Q)$ for $p,q = 0, \dots, n-1$.   By Proposition \ref{prop:rstarrbar},
 $\overline{r}_{p,q}(Q^{\prime}) \ge \overline{r}_{p,q}(Q)$ if and only if
$r^*_{p,q+1}(Q^{\prime}) \ge r^*_{p,q+1}(Q)$ for $p,q =0, \dots, n-1.$  By 
Equations \eqref{eq:standardbruhatorder} and \eqref{eq:starstandard}, $r_{p,q}(Q^{\prime}) \ge r_{p,q}(Q)$ and $r^*_{p,q}(Q^{\prime}) \ge r^*_{p,q}(Q)$ for $p,\, q=1,\dots, n$ if and only if $B\cdot Q^{\prime} \subset \overline{B\cdot Q}$ and $B^* \cdot Q^{\prime} \subset \overline{B^* \cdot Q},$
and this last condition is equivalent to $\Sh(Q^{\prime})\le \Sh(Q).$
\end{proof}

\begin{cor}
Let $Q\in\Borbitspace$ with $Q_{B}=B\cdot Q$ and $Q_{B^{*}}=B^{*}\cdot Q$, so 
that $Q=Q_{B}\cap Q_{B^{*}}$ by Remark \ref{r:geoshar}.  Then 
$\overline{Q}=\overline{Q_{B}}\cap \overline{Q_{B^{*}}}$.  
\end{cor}

\begin{proof}
In the notation of Definition \ref{d:geoShpair}, $\overline{Q} = \overline{Q_B \cap Q_{B^{*}}} \subset \overline{Q_B} \cap \overline{Q_{B^{*}}}.$   Let $Q^{\prime} \subset \overline{Q_B} \cap \overline{Q_{B^{*}}}.$  Then $B\cdot Q^{\prime} \subset \overline{Q_B}$ and $B^{*}\cdot Q^{\prime} \subset \overline{Q_{B^{*}}}.$  Hence,
$\Sh(Q^{\prime}) \le \Sh(Q),$ so that $Q^{\prime} \subset \overline{Q}$ by Theorem \ref{t:sharequalbruhat}, and the Corollary follows.
\end{proof}

\begin{rem}\label{r:Magyarmore}
In \cite{Magyar}, the author describes the closure ordering on the entire set 
of $B$-orbits on $\B_{n}\times\mathbb{P}^{n-1}$, whereas our work in this paper applies only to $B$-orbits on $\B_{n}\times\calO_{n}$, with $\calO_{n}=B\cdot [e_{n}]$.  In a subsequent paper, we will describe the closure ordering $B\backslash (\B_{n}\times \mathcal{O}_{i}$) where $\mathcal{O}_{i}=B\cdot[e_{i}]$ for any $i=1,\dots, n$ using an analogue of Shareshian pairs and a generalization of Theorem \ref{t:sharequalbruhat} and use this to describe the closure relations on $B\backslash (\B_{n}\times\mathbb{P}^{n-1}).$  

\end{rem}

Using the parameterization of $\Borbitspace$ by Shareshian pairs, we can understand the structure of $\Borbitspace$ in a simpler manner than in \cite{CE21I} and \cite{CE21II} or \cite{Magyar}.  In the next section, we show that the extended monoid action on $\Borbitspace$ is given by extending a version of the left and right monoid actions on $W$ to the product $W\times W$ in a manner which preserves Shareshian pairs.


\section{The Monoid Action and Shareshian Pairs}\label{s:monoid}

We now show the extended monoid action by simple roots of $\fg$ and $\fk:=\fg_{n-1}$ on $\Borbitspace$ studied in \cite{CE21I} and \cite{CE21II} can be computed using a monoid action on the set $\Sp$ of all Shareshian pairs.  We give a formula for the dimension of a $B_{n-1}$-orbit in terms of its Shareshian pair.  We begin by recalling the construction of a monoid action on the set of orbits of an algebraic subgroup $M$ of an algebraic group $R$ acting on the flag variety $\B_{R}$ of $R$  with finitely many orbits.  

\subsection{Background on Monoid actions}\label{ss:monoidbackground}
For more details on the subsequent material, we refer the reader to \cite{RS}, \cite{Vg}, \cite{CEexp}, \cite{CE21I} and other sources.
Let $R$ be a connected reductive algebraic group, let $\B=\B_{R}$ be the flag 
variety of $R$, and let $M$ be an algebraic subgroup of $R$ acting on $\B$ with finitely many orbits.  Identify $\B \cong R/B_{R}$, for a Borel subgroup $B_{R}\subset R$ and let $\fb_{\fr}=\mbox{Lie}(B_{R})\subset\fr$.  
Let $\Pi_{\fr}$ be the set of simple roots defined by the Borel subalgebra $\fb_{\fr}$, and let $S_{R}$ be the simple reflections of the Weyl group $W$ of $R$ corresponding to $\Pi_{\fr}$.  For $\alpha\in \Pi_{\fr}$, let $\mathcal{P}_{\alpha}$ be the variety of all parabolic subalgebras of $\fr$ of type $\alpha$ and consider the $\mathbb{P}^{1}$-bundle $\pi_{\alpha}:\B \to {\mathcal{P}}_{\alpha}$.  
For $\alpha \in \Pi_{\fr}$ with corresponding reflection $s=s_{\alpha} \in W$, we define an operator $m(s)$ on the set of orbits $M\backslash \B$ following the above sources.  For $Q_M \in M\backslash \B$, let $m(s)*Q_{M}$ be the unique $M$-orbit which is open and dense in $\pi_{\alpha}^{-1}(\pi_{\alpha}(Q_{M})).$   Note that $\pi_{\alpha}: \pi_{\alpha}^{-1}(\pi_{\alpha}(Q_{M}))\to \pi_{\alpha}(Q_{M})$ is an $M$-equivariant 
$\mathbb{P}^{1}$-bundle. Thus, $\dim(\pi_{\alpha}^{-1}(\pi_{\alpha}(Q_{M})))=\dim(\pi_{\alpha}(Q_{M}))+1,$ and since $Q_M \subset \pi_{\alpha}^{-1}(\pi_{\alpha}(Q_{M}))$, it follows that the orbit $Q_{M} = m(s)*Q_{M}$ if and only if $\dim(m(s)*Q_{M})=\dim(Q_{M}).$

\begin{rem}\label{r:monoidmove}
It follows that  $Q_{M} \not= m(s)*Q_{M}$ if and only if $\dim(m(s)*Q_{M})=\dim(Q_{M})+1.$
\end{rem}

Computation of $m(s)*Q_{M}$ depends on the \emph{the type of the root} $\alpha$ for the orbit $Q_{M}$, which is determined as follows.  For $\fb^{\prime}\in Q_M$, let $\fp_{\alpha}^{\prime} = \pi_{\alpha}(\fb^{\prime})$ and let $B^{\prime}$ and $P_{\alpha}^{\prime}$ be the corresponding parabolic subgroups of $R$, and let $V_{\alpha}^{\prime}$ be the the solvable radical of $P_{\alpha}^{\prime}$.  
Consider the group 
$S_{\alpha}^{\prime}:= P_{\alpha}^{\prime}/V_{\alpha}^{\prime}$ isogenous to $SL(2)$ and its subgroup 
$M_{\alpha, \fb^{\prime}}:= (M \cap P_{\alpha}^{\prime})/(M\cap V_{\alpha}^{\prime})$.   
An easy argument using base change for the inclusion $Q \to \B_n$ with respect to the morphism $\B_n \to G/P_{\alpha}^{\prime}$ gives an identification 
\begin{equation}\label{e:monoidid}
\pi_{\alpha}^{-1}(\pi_{\alpha}(Q)) \cong M\times_{M\cap P_{\alpha}^{\prime}} P_{\alpha}^{\prime}/B^{\prime}.
\end{equation}
It follows that $M$-orbits in
 $\pi_{\alpha}^{-1}(\pi_{\alpha}(Q))$ correspond to $M_{\alpha, \fb^{\prime}}$-orbits in 
$P^{\prime}_{\alpha}/B^{\prime}\cong {\PR}^1.$  Thus, $M_{\alpha, \fb^{\prime}}$ has only finitely many orbits on $P_{\alpha}^{\prime}/B^{\prime},$ and thus matches one of the 4 cases from Section 4.1 of \cite{RS}, which we list below in slightly reorganized form.

\begin{dfn}\label{d:roottype}
\par\noindent (1) If $M_{\alpha, \fb^{\prime}}$ is solvable and contains the unipotent radical of a Borel subgroup of $S_{\alpha}^{\prime}$, then $\alpha$ is called a complex root for $Q_M$.  If $M_{\alpha, \fb^{\prime}}\cdot \fb^{\prime}=\fb^{\prime}$, then $\alpha$ is complex stable for $Q_M$ and otherwise $\alpha$ is complex unstable for $Q_M$.
\par\noindent (2) If $M_{\alpha, \fb^{\prime}}=S_{\alpha}^{\prime}$, then $\alpha$ is called a compact   root for $Q_M.$
\par\noindent (3) Suppose $M_{\alpha, \fb^{\prime}}$ is one-dimensional and reductive.
If $M_{\alpha,\fb^{\prime}}\cdot \fb^{\prime}=\fb^{\prime}$, then $\alpha$ is called a noncompact   root for $Q_M$, while if $M_{\alpha,\fb^{\prime}}\cdot \fb^{\prime}\not=\fb^{\prime},$ then $\alpha$ is called a real root for $Q_M.$
\end{dfn}

In \cite{RS}, the monoid action is only discussed when $M$ is locally the fixed points of an involution of $R$, and the terminology in this definition comes from the action of the involution on root spaces, or more precisely the action of an associated real form of $R.$   However, as many authors have observed, including \cite{Knoporbits}, the construction works in the same way when $M$ has finitely many orbits on the flag variety.  It follows easily from arguments in \cite{RScomp} that the monoid action depends only the nature of the groups $M_{\alpha, \fb^{\prime}}.$   This is the perspective we use in this paper.
It is well-known and easy to prove that the type of the root depends only on the orbit $Q_M$ and not on the point $\fb^{\prime}.$ The following Proposition gives more detail in the cases when $m(s)*Q \not= Q.$

\begin{prop}\label{p:stableandnc}(Lemmas 2.1.4 and 2.4.3 of \cite{RScomp})
Let $Q\in M\backslash\B$ with $Q=M\cdot \Ad(v)\fb_{\fr}$ and $\alpha\in \Pi_{\fr}$.
\begin{enumerate}
\item If $\alpha$ is complex stable for $Q$, then $\pi_{\alpha}^{-1}(\pi_{\alpha}(Q))$ consists of two $M$-orbits:
$$
\pi_{\alpha}^{-1}(\pi_{\alpha}(Q))=Q\cup M\cdot \Ad(vs)\fb_{\fr},
$$
and $m(s)*Q= M\cdot \Ad(vs)\fb_{\fr}$. 
\item If $\alpha$ is non-compact   for $Q$ then $\pi_{\alpha}^{-1}(\pi_{\alpha}(Q))$ consists of two (resp. three) $M$-orbits, depending on whether $M_{\alpha, \fb^{\prime}}$ is the normalizer of a torus (resp. a torus). The open orbit
$m(s)*Q=M\cdot \Ad(vu_{\alpha})\fb_{\fr}$, where $u_{\alpha}\in R$ is the Cayley transform with respect to the root $\alpha$ as defined in Equation (41) of \cite{CEexp} and $\dim Q=\dim (M\cdot \Ad(vs)\fb_{\fr})$.   

\end{enumerate}

\end{prop}


\subsection{Extended Monoid Action on $\Borbitspace$}\label{ss:monoidactBn-1}
Recall that we set $K:=G_{n-1}$.  In this subsection, we apply the above construction in the case where $R=K\times G$ and $M=K_{\Delta}=\{(x,x):\, x\in K\}$ is the diagonal copy of $K$ in the product $K\times G$.  The discussion above yields a monoid action by the standard simple roots of $\fk\oplus\fg$, $\Pi_{\fk\oplus\fg}=\Pi_{\fk}\sqcup \Pi_{\fg}$, on the set of orbits $K_{\Delta}\backslash (\B_{n-1}\times \B).$  There is a one-to-one correspondence 
\begin{equation}\label{eq:orbitbijection}
\Borbitspace \leftrightarrow K_{\Delta}\backslash (\B_{n-1}\times \B) \mbox{ given by } Q\leftrightarrow K_{\Delta}\cdot (\fb_{n-1}, Q).
\end{equation} 
For $Q\in\Borbitspace$, let $Q_{\Delta} = K_{\Delta}\cdot (\fb_{n-1}, Q)$ be the corresponding $K_{\Delta}$-orbit.  Note that $Q_{\Delta} \cong K\times_{B_{n-1}} Q$, and it follows that the map $Q \mapsto Q_{\Delta}$
preserves topological properties like closure relations and open sets, and $\dim(Q_{\Delta})=\dim(Q) + \dim(\B_{n-1}).$   The correspondence in (\ref{eq:orbitbijection}) allows us to transfer the monoid action of $\Pi_{\fk} \sqcup \Pi_{\fg}$ on $K_{\Delta}\backslash (\B_{n-1}\times \B)$ to a monoid action of $\Pi_{\fk} \sqcup \Pi_{\fg}$ on $\Borbitspace$, which we refer to as the extended monoid action.  In more detail, for $v\in G$, let $Q=B_{n-1}\cdot\Ad(v)\fb \in\Borbitspace.$   Let $\alpha \in \Pi_{\fk}.$  By Equation (4.1) of \cite{CE21I}, 
\begin{equation}\label{eq:leftmonoid}
m(s_{\alpha})*Q \mbox{ is the unique open $B_{n-1}$-orbit in } P^{K}_{\alpha} \cdot \Ad(v)\fb,
\end{equation}
where $P^{K}_{\alpha}\supset B_{n-1}$ is the standard parabolic subgroup of $K$ determined by the root $\alpha$.  We refer to this monoid action as ``the left monoid action''.  Let $\alpha \in \Pi_{\fg}.$  By Equation (4.2) of \cite{CE21I}, 
\begin{equation}\label{eq:rightmonoid}
m(s_{\alpha})*Q\mbox{ is the unique open $B_{n-1}$-orbit in } 
B_{n-1} \cdot \Ad(vP_{\alpha})\fb,
\end{equation}
where $P_{\alpha}\supset B$ is the standard parabolic subgroup of $G$ determined by the root $\alpha$.   We refer to this monoid action as ``the right monoid action''.  

We describe the groups $(K_{\Delta})_{\alpha, \fb^{\prime}}$ for the left and right monoid actions in more detail.   First, let $\alpha \in \Pi_{\fk},$  
and as in Subsection \ref{ss:monoidbackground}, let $Q=B_{n-1}\cdot\Ad(v)\fb$ so that $Q_{\Delta}=K_{\Delta}\cdot(\fb_{n-1},\Ad(v)\fb)$.  Then with $\fb^{\prime}=(\fb_{n-1},\Ad(v)\fb),$ the group
\begin{equation}\label{eq:localKkroot}
(K_{\Delta})_{\alpha, \fb^{\prime}}=(K_{\Delta}\cap (P^{K}_{\alpha},\Ad(v)B))/(K_{\Delta}\cap(V_{\alpha}^{K}, \Ad(v)B))
\cong (P^{K}_{\alpha} \cap \Ad(v)B)/(V^{K}_{\alpha} \cap \Ad(v)B),
\end{equation}
where $V_{\alpha}^{K}$ is the solvable radical of $P_{\alpha}^{K}$.  
 Further, $m(s_{\alpha})*Q$ corresponds to the open $(P^{K}_{\alpha} \cap \Ad(v)B)/(V^{K}_{\alpha} \cap \Ad(v)B)$-orbit on $P_{\alpha}^K/B_{n-1}.$  Now suppose $\alpha \in \Pi_{\fg}$ and for $Q$ as above, the corresponding group is
\begin{equation}\label{eq:localKgroot}
(K_{\Delta})_{\alpha, \fb^{\prime}}=(K_{\Delta}\cap (B_{n-1},\Ad(v)P_{\alpha}))/(K_{\Delta} \cap(B_{n-1},\Ad(v)V_{\alpha}) ) \cong (B_{n-1}\cap \Ad(v)P_{\alpha})/(B_{n-1}\cap \Ad(v)V_{\alpha}),
\end{equation}
where $V_{\alpha}$ is the solvable radical of $P_{\alpha}.$
%
Further, $m(s_{\alpha})*Q$ corresponds to the open $(B_{n-1}\cap \Ad(v)P_{\alpha})/(B_{n-1}\cap \Ad(v)V_{\alpha})$-orbit on $(\Ad(v)P_{\alpha})/(\Ad(v)B).$

\begin{rem}\label{r:ncompact}  Note that for $\alpha \in \Pi_{\fk}$ or $\Pi_{\fg},$ the group $(K_{\Delta})_{\alpha,\fb^{\prime}}$ is solvable. It follows  that $\alpha$ is never a compact root for $Q$ (see Definition \ref{d:roottype}(2)) and there must be $3$ orbits in the non-compact case by Proposition \ref{p:stableandnc}(2).
\end{rem}

\begin{rem}\label{r:monoidequiv}
The right monoid action also arises by applying the construction from the last subsection to the case where $M=B_{n-1}$ and $R=G.$   The left monoid action  arises from a variant of the construction of the last subsection by using the fibre bundle $B_{n-1}\backslash G \to P^{K}_{\alpha}\backslash G.$
\end{rem}


\subsection{Monoid action on Shareshian Pairs}\label{ss:Shmonoidact}
As in the last subsection, the general theory in Subsection \ref{ss:monoidbackground} can also be applied in the case where $R=G\times G$ and $M=G_{\Delta}$ is the diagonal copy of $G$ in $G\times G$.  In this case, by realizing $\B\cong G/B$, we recover the standard bijection 
\begin{equation}\label{e.gdiagschubert}
 B\backslash G/B \cong G_{\Delta}\backslash(G/B\times G/B), \ BwB/B \mapsto G_{\Delta}\cdot (B,wB),
\end{equation}
 and the orbits are indexed by $w\in W.$    By applying the orbit correspondence from (\ref{e.gdiagschubert}), we can transfer the monoid action by the simple roots $\Pi_{\fg\oplus\fg}=\Pi_{\fg}\sqcup\Pi_{\fg}$ to left and right monoid actions of $\Pi_{\fg}$  on $B\backslash G/B.$  The left monoid action is given by a similar formula to  (\ref{eq:leftmonoid}), and the right monoid action is given as in (\ref{eq:rightmonoid}), with the role of $B_{n-1}$ in both equations replaced by $B$ and the role of $P^{K}_{\alpha}$ in (\ref{eq:leftmonoid}) replaced by the standard parabolic subgroup $P_{\alpha}$ that appears in (\ref{eq:rightmonoid}).  It follows that $m(s_{\alpha})*_L BwB/B$ is the unique open $B$-orbit in $P_{\alpha} wB/B$, and by basic results on the Bruhat decomposition, this is $Bs_{\alpha} w B/B$ if $\ell(s_{\alpha} w) > \ell(w)$, and otherwise is $BwB/B.$
 Similar results apply to the right monoid action, except in the conclusion $ws_{\alpha}$ replaces $s_{\alpha} w.$ 

For $\alpha$ in the left copy of $\Pi_{\fg},$ and the orbit $Q_{\Delta}$ through $\fb^{\prime}=(\fb,\Ad(x)\fb)$ with $x\in G,$ the group
\begin{equation}\label{e:Galphaleftgroup}
(G_{\Delta})_{\alpha, \fb^{\prime}}=G_{\Delta} \cap (P_{\alpha},\Ad(x)B)/(V_{\alpha},\Ad(x)B) \cong (P_{\alpha} \cap \Ad(x)B)/(V_{\alpha}\cap \Ad(x)B).
\end{equation}
and $\ms * Q$ corresponds to the open $(G_{\Delta})_{\alpha, \fb^{\prime}}$-orbit 
in $P_{\alpha}/B.$
For $\alpha$ in the right copy of $\Pi_{\fg}$ and the orbit $Q_{\Delta}$ as above,
the group
\begin{equation}\label{e:Galpharightgroup}
(G_{\Delta})_{\alpha, \fb^{\prime}}=(G_{\Delta} \cap (B,\Ad(x)P_{\alpha}))/(B,\Ad(x)V_{\alpha}) \cong (B \cap \Ad(x)P_{\alpha})/(B \cap \Ad(x)V_{\alpha}),
\end{equation}
and $\ms * Q$ corresponds to the open $(G_{\Delta})_{\alpha, \fb^{\prime}}$-orbit in $\Ad(x)P_{\alpha}/\Ad(x)V_{\alpha}.$
 Using the fact that we can assume $x\in W$, we note that in each case $(G_{\Delta})_{\alpha, \fb^{\prime}}$ is a Borel subgroup of the group $S^{\prime}_{\alpha},$ and from this we recover the well-known fact that
 for each $B$-orbit $Q_{B}$ in $G/B$ and each simple root $\alpha\in\Pi_{\fg \oplus \fg}$, the root $\alpha$ is either complex stable or unstable for $Q_{B}$. 

We now describe these monoid actions by using a Weyl group element $w$ to represent the corresponding Schubert cell $BwB/B.$  The conclusion is that for each simple root $\alpha,$
\begin{equation}\label{eq:Wmonoidact}
\begin{split}
\mbox{(Left Action) } &m(s_{\alpha})*_{L} w=w \mbox{ if } \ell(s_{\alpha}w)<\ell(w) \mbox{ and } m(s_{\alpha})*_{L}w=s_{\alpha}w\mbox{ if } \ell(s_{\alpha}w)>\ell(w).\\
\mbox{(Right Action) } &m(s_{\alpha})*_{R}w=w \mbox{ if } \ell(ws_{\alpha})<\ell(w) \mbox{ and } m(s_{\alpha})*_{R}w=ws_{\alpha}\mbox{ if } \ell(ws_{\alpha})>\ell(w).\\
\end{split}
\end{equation}

\begin{nota}\label{nota:Worders}
By considering the simple roots $\Pi_{\fg}^{*}$ and corresponding simple reflections $S^{*}$, we obtain a new order relation on $W$ by considering the length of an element $w^{*}\in W$ with respect to the set of simple reflections $S^{*}$.  We denote $W$ with this non-standard order relation by $(W, S^{*})$, and denote $W$ with the standard order relation by $(W, S)$.  We denote elements in the poset $(W,S)$ by $w$ and denote elements in the poset $(W, S^{*})$ by $w^{*}$.
\end{nota}

The same construction can be repeated replacing the Borel subgroup $B$ with
the Borel subgroup $B^{*}$ and the simple roots $\Pi_{\fg}$ with the set of simple roots $\Pi_{\fg}^{*}$ and using the poset $(W, S^{*})$ described in Notation \ref{nota:Worders}.  The monoid actions are then given by (\ref{eq:Wmonoidact}) with $\alpha\in\Pi_{\fg}$ replaced by $\alpha^{*}\in\Pi_{\fg}^{*}$ and with $\ell(w^{*})$ denoting the length of $w^{*}$ as an element of $(W,S^{*})$.  

We now show that the extended monoid action on $\Borbitspace$ discussed in Section \ref{ss:monoidactBn-1} can be computed using the Shareshian map and the monoid actions on the Weyl group $W$ given in (\ref{eq:Wmonoidact}).  
In more detail, Equation (\ref{eq:Wmonoidact}) implies that we obtain both left and right monoid actions via simple roots of $\fg$ on each factor of the product $(W,S)\times (W,S^{*})$.  We can restrict this monoid action to a left action via simple roots of $\fk$ and a right action via simple roots of $\fg$ on the product $(W,S)\times (W,S^{*})$ as follows.  Recall the $n$-cycle $\sigma=(n,n-1,\dots, 2, 1)$ from Proposition \ref{p:Shpairs}.  Given $\alpha\in\Pi_{\fg}$ let $\alpha^* = \sigma(\alpha)$ be the corresponding simple root of $\Pi_{\fg}^{*}$.   The standard simple roots of $\fk$ can be viewed as a subset of $\Pi_{\fg}^{*}$ by observing that $\Pi_{\fk}=\{\alpha_{2}^{*},\dots, \alpha_{n-1}^{*}\}$.


\begin{dfn}\label{d:diagonal}
Define the \emph{restricted diagonal monoid action} on $(W,S)\times (W,S^{*})$ via simple roots $\Pi_{\fk \oplus \fg}$ as follows. 

  
For $(w,u^{*})\in (W,S)\times (W,S^{*})$, define
\begin{equation}\label{eq:diagonalmonoid}
\begin{split}
\mbox{ (Left action)} &\mbox{ For } \alpha\in\Pi_{\fk},\, m(s_{\alpha})*_{L}(w,u^{*})=(m(s_{\alpha})*_{L}w, m(s_{\alpha})*_{L}u^{*}).\\
\mbox{ (Right action)} &\mbox{ For } \alpha\in\Pi_{\fg}, m(s_{\alpha})*_{R}(w,u^{*})=(m(s_{\alpha})*_{R}w, m(s_{\alpha^{*}})*_{R}u^{*}),
\end{split}
\end{equation}
where the monoid actions $*_{L}$ and $*_{R}$ are given in Equation (\ref{eq:Wmonoidact}).
\end{dfn}
Henceforth, we will drop the subscripts $L$ and $R$ on the monoid actions defined in Equation (\ref{eq:diagonalmonoid}) and use the convention that a simple root $\alpha\in\Pi_{\fk \oplus \fg}$ acts on the left whenever $\alpha\in \Pi_{\fk}$ and on the right if $\alpha\in\Pi_{\fg}$.  We now arrive at the main result of this section.  


\begin{thm}\label{thm:intertwine}
The Shareshian map $\Sh: \Borbitspace \to \Sp\subset W\times W$ given in Equation (\ref{eq:Shmap})
intertwines the extended monoid action on $\Borbitspace$ given in Equations (\ref{eq:leftmonoid}) and (\ref{eq:rightmonoid}) with the restricted diagonal monoid action on $(W,S)\times (W,S^{*})$ given in (\ref{eq:diagonalmonoid}), i.e., for $Q\in\Borbitspace$ and $\alpha\in\Pi_{\fk \oplus \fg}$,
 \begin{equation}\label{eq:Shinter}
 \Sh(m(s_{\alpha})*Q)=m(s_{\alpha})*\Sh(Q).  
 \end{equation}
In particular, the restricted diagonal monoid action on $(W,S)\times (W, S^{*})$ preserves the set $\Sp$ of all Shareshian pairs.  
Moreover, if $m(s_{\alpha})*Q\neq Q$, then the type of the root $\alpha$ is determined by the type of $\alpha$ for the corresponding Shareshian pair $\Sh(Q)=(w,u^{*})$.  More precisely, for a root $\alpha\in\Pi_{\fk \oplus \fg}$, 
\begin{enumerate}
\item The root $\alpha$ is complex stable for $Q$ if and only if it is complex stable for both $w$ and $u^{*}$.
\item The root $\alpha$ is non-compact for $Q$ if and only if $\alpha$ is complex stable for exactly one of $w$ and $u^{*}$ and unstable for the other. 
\item The root $\alpha$ is real or complex unstable for $Q$ if and only if $\alpha$ is complex unstable for both $w$ and $u^{*}$.
\end{enumerate}
\end{thm}
\begin{proof}

Let $\Sh(Q)=(Q_B, Q_{B^*})\in B\backslash\B_{n}\times B^{*}\backslash\B_{n}$ so that $Q=Q_B \cap Q_{B^*}$ by Remark \ref{r:geoshar}.
We claim that for $\alpha\in \Pi_{\fk}$,  
\begin{equation}\label{eq:leftmsainter}
\ms*Q=\ms*Q_{B}\cap\ms*Q_{B^{*}},
\end{equation}
and for $\alpha \in \Pi_{\fg}$,
\begin{equation}\label{eq:rightmsainter}
\ms*Q=\ms*Q_{B}\cap m(s_{\alpha^*})*Q_{B^{*}}.
\end{equation}
Equation (\ref{eq:Shinter}) follows from Equations (\ref{eq:diagonalmonoid}), (\ref{eq:leftmsainter}) and (\ref{eq:rightmsainter}) and Remark \ref{r:geoshar}.  Let 
$Q=B_{n-1}xB/B$, so that $Q_{B}=BxB/B$ and $Q_{B^{*}}=B^{*}xB/B.$  
We first prove (\ref{eq:leftmsainter}) by adapting the proof of Proposition 4.7 in \cite{CE21I} to the left monoid action.  We claim that for $\alpha \in \Pi_{\fk}$,
\begin{equation}\label{eq:firstinc}
\ms*Q\subset \ms*Q_{B}.
\end{equation}
By Equation (\ref{eq:leftmonoid}), $\ms*Q$ is the open $B_{n-1}$-orbit in $P_{\alpha}^{K} xB/B$, and similarly
$\ms*Q_{B}$ is the open $B$-orbit in $P_{\alpha} xB/B$.  Note that 
$P_{\alpha}^{K} x B/B\subseteq P_{\alpha} x B/B$.  We now show that 
\begin{equation}\label{eq:nonempty}
\ms*Q_{B}\cap P_{\alpha}^{K} x B/B\neq\emptyset.
\end{equation}
 If $\alpha$ is complex unstable for $Q_{B}$, then $\ms*Q_{B}=Q_{B}=BxB/B$ and the assertion is clear.  On the other hand, if $\alpha$ is complex stable, then $\ms*Q_{B}=B s_{\alpha} xB/B$.  Since $\alpha\in \Pi_{\fk}$, we can take a representative ${\dot{s}}_{\alpha}$ for $s_{\alpha}$ to be in the group $P_{\alpha}^{K}$, so that ${\dot{s}}_{\alpha}xB/B\in P_{\alpha}^{K}xB/B\cap \ms*Q_{B}$. Since $\alpha$ is complex for $Q_B$, these two cases show that Equation (\ref{eq:nonempty}) holds.  Since $P_{\alpha}^{K}xB/B$ is irreducible, it follows that $\ms*Q\cap\ms*Q_{B}\neq\emptyset$, so that $(\ms*Q)_{B}= B\cdot(\ms*Q)=\ms*Q_{B}$.  One can repeat the argument above with $Q_{B}$ replaced by $Q_{B^{*}}$ and $P_{\alpha}$ replaced by $P_{\alpha^{*}}$ to obtain $(\ms*Q)_{B^{*}}= \ms*Q_{B^{*}}$.  Equation (\ref{eq:leftmsainter}) follows.

 We now prove Equation (\ref{eq:rightmsainter}).  We first prove Equation (\ref{eq:firstinc}) for $\alpha\in\Pi_{\fg}$, again following the proof of Proposition 4.7 of \cite{CE21I}.   The orbit $\ms*Q$ is the open $B_{n-1}$-orbit 
 in $B_{n-1}x P_{\alpha}/B$, and $\ms*Q_{B}$ is the open $B$-orbit in $BxP_{\alpha}/B$.  Depending on whether $Q_B$ is complex stable or complex unstable for $\alpha$, $xB/B$ or $xs_{\alpha}B/B$ is in $\ms*Q_B$.    Since $B_{n-1} xP_{\alpha} /B\subset B x P_{\alpha}/B$, the subvariety $\ms*Q_{B}\cap B_{n-1}xP_{\alpha}/B$ is thus an open, nonempty subvariety of $B_{n-1}xP_{\alpha}/B$.  As above, it follows that $\ms*Q\cap\ms*Q_{B}\neq\emptyset$, which implies Equation (\ref{eq:firstinc}).
 To prove that $\ms*Q\subset m(s_{\alpha^{*}})*Q_{B^{*}}$, note that $Q=B_{n-1}x\cdot \fb =B_{n-1}x\sigma^{-1}\cdot {\fb}^{*}$, where $\sigma=(n, n-1,\dots, 2,1)\in W$.  Then $Q_{B^{*}}=B^{*} x\sigma^{-1}\cdot\fb^{*}$, so that $m(s_{\alpha^{*}})*Q_{B^{*}}$ is the open $B^{*}$-orbit in $B^{*}x\sigma^{-1} P_{\alpha^{*}}\cdot \fb^{*}$.  Since $P_{\alpha^{*}}=\Ad(\sigma)P_{\alpha}$, we have 
 \begin{equation}\label{eq:unwind}
 B^{*} x\sigma^{-1} P_{\alpha^{*}} \cdot \fb^*=B^{*}x\sigma^{-1}\sigma P_{\alpha}\cdot \fb=B^{*}xP_{\alpha} \cdot \fb.
 \end{equation}
It now follows that $\ms*Q\subset m(s_{\alpha^{*}})*Q_{B^{*}}$ using the same argument given above with the $B$-action on $\B_n$ replaced by the $B^{*}$-action on $\B_n$.  We therefore obtain Equation (\ref{eq:rightmsainter}) for $\alpha\in \Pi_{ \fg}$.  

To prove part (3) of Theorem \ref{thm:intertwine}, first assume $\alpha \in \Pi_{\fk}$ is complex unstable or real for $Q$ so that $\ms*Q=Q.$   By Remark \ref{r:geoshar} and Equation (\ref{eq:leftmsainter}), $\ms*Q_B=Q_B,$ and  $\ms*Q_{B^*}=Q_{B^*},$ so that $\alpha$  is complex unstable for $Q_B$ and for $Q_{B^*}.$  The case where $\alpha \in \Pi_{\fg}$ follows similarly using Equation (\ref{eq:rightmsainter}), and the converse of (3) follows by reversing these arguments.   Part (2) of the Theorem follows from parts (1) and (3) of the Theorem and Remark \ref{r:ncompact}.

It remains to prove part (1) of the Theorem and for this, we first suppose $\alpha\in \Pi_{\fk}$ and that $\alpha$ is complex stable for both $Q_{B}$ and $Q_{B^{*}}$.  Then by Equation (\ref{eq:Wmonoidact}), $\ms*Q_{B}=Bs_{\alpha}x B/B$ and $\ms*Q_{B^{*}}=B^{*}s_{\alpha}xB/B.$  It follows from Equation (\ref{eq:leftmsainter}) that $\ms*Q\neq Q$ and that $\ms*Q=B_{n-1}s_{\alpha} x B/B$.  We claim this implies that $\alpha$ is complex stable for $Q$.  Indeed, since $\ms*Q\neq Q$, $\alpha$ is either complex stable or non-compact   for $Q$.  If $\alpha$ were non-compact for $Q$, then $\dim(B_{n-1} s_{\alpha}xB/B)=\dim Q$ by Proposition \ref{p:stableandnc}(2).  However, by Remark \ref{r:monoidmove}, $\dim (B_{n-1} s_{\alpha}xB/B)=\dim Q+1$, so $\alpha$ must be complex stable for $Q.$  

Now suppose that $\alpha\in \Pi_{\fk}$ is complex stable for $Q=B_{n-1}xB/B$ and consider the point $\fb^{\prime}=(\fb_{n-1},\Ad(x)\fb)$ and its $K_{\Delta}$-orbit $Q_{\Delta}=K_{\Delta}\cdot \fb^{\prime}.$  The corresponding $G_{\Delta}$-orbit for $Q_B$ is $G_{\Delta}\cdot \fb^{\prime\prime}$ where $\fb^{\prime\prime}= (\fb,\Ad(x)\fb).$ 
The minimal parabolic for $\alpha$ relative to the point $\fb^{\prime}$ is $(P_{\alpha}^{K},\Ad(x)B),$ and the minimal parabolic for $\alpha$ relative to the point $\fb^{\prime\prime}$ is $(P_{\alpha}, \Ad(x)B).$  Since $\alpha$ is complex stable for $Q$, it follows from Definition \ref{d:roottype} that $(K_{\Delta})_{\alpha,\fb^{\prime}}$ contains the unipotent radical of a Borel subgroup of $S_{\alpha}^{K,\prime}=(P_{\alpha}^{K},\Ad(x)B)/(V_{\alpha}^{K},\Ad(x)B)$ and $(K_{\Delta})_{\alpha,\fb^{\prime}}\cdot\fb^{\prime}=\fb^{\prime}.$  We first show that $\alpha$ is complex stable for $Q_B=B\cdot \Ad(x)\fb.$   It follows by Equations (\ref{eq:localKkroot}) and (\ref{e:Galphaleftgroup}) that $(K_{\Delta})_{\alpha, \fb^{\prime}}$ embeds as a subgroup of $(G_{\Delta})_{\alpha, \fb^{\prime\prime}}.$   Moreover, the equivariant embedding of the subvariety $(P_{\alpha}^{K},\Ad(x)B)\cdot \fb^{\prime}\cong P_{\alpha}^K/B_{n-1}$ into $(P_{\alpha},\Ad(x)B)\cdot \fb^{\prime\prime}\cong P_{\alpha}/B$ mapping $\fb^{\prime}$ to $\fb^{\prime\prime}$ is  an isomorphism.  It follows that $(K_{\Delta})_{\alpha,\fb^{\prime}}$ fixes $\fb^{\prime\prime},$ and contains the unipotent radical of a Borel subgroup of the corresponding group $S_{\alpha}^{\prime\prime}=(P_{\alpha},\Ad(x)B)/(V_{\alpha},\Ad(x)B).$    Since $(G_{\Delta})_{\alpha, \fb^{\prime\prime}}$ is solvable, it follows that $(G_{\Delta})_{\alpha, \fb^{\prime\prime}}$ also fixes $\fb^{\prime\prime}$, and that $\alpha$ is complex stable for $Q_B.$   To prove that $\alpha$ is complex stable for $Q_{B^*},$ we do the same analysis with $B^*$ and $P_{\alpha^*}$ playing the roles of $B$ and $P_{\alpha}.$

We now prove Part (1) of Theorem \ref{thm:intertwine} for the right action.  
Suppose $\alpha \in \Pi_{\fg}$ and note that the minimal parabolic for $\alpha$ relative to $\fb^{\prime}$ is $(B_{n-1},\Ad(x)P_{\alpha})$, and the minimal parabolic for $\alpha$ relative to $\fb^{\prime\prime}$ is $(B,\Ad(x)P_{\alpha}).$  
  Now suppose $\alpha$ is complex stable for both $Q_{B}$ and $Q_{B^{*}}$.   By Equation (\ref{eq:rightmonoid}) and Propostion \ref{p:stableandnc}(1), we deduce that $\ms*Q_B=Bxs_{\alpha}B/B$.   By the argument used to justify
Equation (\ref{eq:unwind}), we conclude that $m(s_{\alpha^{*}})*Q_{B^{*}}=B^{*}xs_{\alpha}B/B$.   By Equation (\ref{eq:rightmsainter}), it follows that 
$\ms*Q=B_{n-1}xs_{\alpha}B/B$, and the rest of the argument that $\alpha$ is complex stable for $Q$ follows in the same way as when $\alpha \in \Pi_{\fk}.$
For the converse, suppose that $\alpha\in\Pi_{\fg}$ is complex stable for $Q=\Ad(x)\fb.$ 
 Recall that $\ms*Q$ corresponds to the open $(K_{\Delta})_{\alpha, \fb^{\prime}}$-orbit on $(B_{n-1},\Ad(x)P_{\alpha})\cdot \fb^{\prime},$ while $\ms*Q_{B}$ corresponds to the open $(G_{\Delta})_{\alpha,\fb^{\prime\prime}}$-orbit on $(B,\Ad(x)P_{\alpha})\cdot\fb^{\prime\prime}.$
 Since $\alpha$ is complex stable for $Q$, then 
$(K_{\Delta})_{\alpha,\fb^{\prime}}\cdot \fb^{\prime}=\fb^{\prime}$ and
 $(K_{\Delta})_{\alpha,\fb^{\prime}}$ contains the unipotent radical of a Borel subgroup of $S_{\alpha}^{\prime}=(B_{n-1},\Ad(x)P_{\alpha})/(B_{n-1},\Ad(x)V_{\alpha}).$ 
The groups $(K_{\Delta})_{\alpha, \fb^{\prime}}$ and  $(G_{\Delta})_{\alpha, \fb^{\prime\prime}}$ are computed in Equations (\ref{eq:localKgroot}) and (\ref{e:Galpharightgroup}) and from this it follows that
$(K_{\Delta})_{\alpha,\fb^{\prime}}$ embeds as a subgroup in $(G_{\Delta})_{\alpha, \fb^{\prime\prime}}.$
Note also that the equivariant embedding $(B_{n-1},\Ad(x)P_{\alpha})\cdot \fb^{\prime}$ into $(B,\Ad(x)P_{\alpha})\cdot \fb^{\prime\prime}$ taking $\fb^{\prime}$ to $\fb^{\prime\prime}$ is an isomorphism. It follows that 
$(K_{\Delta})_{\alpha,\fb^{\prime}}\cdot \fb^{\prime\prime}=\fb^{\prime\prime}.$   Since $(G_{\Delta})_{\alpha, \fb^{\prime\prime}}$ is solvable, it follows that $(G_{\Delta})_{\alpha,\fb^{\prime\prime}}\cdot\fb^{\prime\prime}=\fb^{\prime\prime}$ just as in the case of the left action above.  Hence, $\alpha$ is complex stable for $Q_B.$  The proof that $\alpha^{*}=\sigma(\alpha)$ is complex stable for $Q_{B^{*}}$ is  analogous, and uses  the fact that $m(s_{\alpha^{*}})*Q_{B^{*}}$ is the open $B^{*}$-orbit in $B^{*}x P_{\alpha}/B$ by Equation (\ref{eq:unwind}).
\end{proof}

As was mentioned in the Introduction, the extended monoid action on $B_{n-1}\backslash\B$ allows us to establish a simple formula for the dimension of an orbit $Q\in\Borbitspace$ (see Theorem \ref{thm:introdim}).  To prove this formula, we need to understand the action of a root $\alpha\in \Pi_{\fk \oplus \fg}$ on $\Sh(Q)=(w,u^{*})$ in the case where $\alpha$ is non-compact   in more detail.  Recall the classification of Shareshian pairs in Proposition \ref{p:Shpairs} and their associated decreasing sequences.  
\begin{prop}\label{p:shnc}
Let $\alpha\in\Pi_{\fk \oplus \fg}$ be non-compact for $Q$ with $\Sh(Q)=(w,u^{*})$.  Let $\Sh(\ms*Q)=(y,v^{*}).$  
Let $\Delta = \Delta(w,u^*)$ and $\Delta^{\prime}=\Delta(y,v^*)$.  Then the cardinality $|\Delta^{\prime}|=|\Delta|+1.$
\end{prop}

We will need the following observation to prove Proposition \ref{p:shnc}.  Recall the description of the linear functionals $\eps_{i}\in\fh^{*}$ at the beginning of Section \ref{s:notation}.
\begin{lem}\label{l:bothindelta}
Let $\Sh(Q)=(w,u^{*})$ with $\Delta=\Delta(w,u^*)$ the associated decreasing sequence.  
\begin{enumerate}
\item Let $\alpha=\eps_{j}-\eps_{j+1}\in \Pi_{\fg}$ for some $j\in\{1,\dots, n-1\}$.  Suppose that $\{w(j),\,w(j+1)\}\subset\Delta$.  Then $\alpha$ is complex unstable for both 
$w$ and $u^{*}$.
\item Let $\alpha=\eps_{i}-\eps_{i+1}\in \Pi_{\fk}$ for some $i \in \{1,\dots, n-2\}$.  Suppose that $\{ i, i+1 \}\subset\Delta$.  Then $\alpha$ is complex unstable for both 
$w$ and $u^{*}$.  
\end{enumerate}
\end{lem}
\begin{proof}
Let $\alpha=\eps_{j}-\eps_{j+1}$ with $j\in\{1,\dots, n-1\}$ be a root of $\fg$.  
Since $w(j)$ and $ w(j+1)$ are in $\Delta$ then Part (2) of Proposition \ref{p:Shpairs} implies that 
$w(j+1)<w(j)$ which is equivalent to $ws_{\alpha}<w$, so that $\alpha$ is complex unstable for $w$.  To describe the Bruhat order on $(W,S^{*})$, we 
 introduce a new total order $\prec$ on the set $\{1,\dots, n\}$ by declaring $n\prec 1\prec 2\dots \prec n-1$.  
The simple reflection $s_{\alpha^{*}}=\sigma s_{\alpha} \sigma^{-1}=\eps_{\sigma(j)}-\eps_{\sigma(j+1)}$, so that 
$u^{*}s_{\alpha^{*}}<u^{*}$ if and only if $u^{*}(\sigma(j+1))\prec u^{*}(\sigma(j)) .$ 
Let $\Delta=\{\ell_{1}<\dots <\ell_{m-1}< \ell_{m}<\ell_{m+1}<\dots <\ell_{k}=n\}$.  Then since 
$\Delta$ is a decreasing sequence for $w^{-1}$, $w(j+1)=\ell_{m}$ and $w(j)=\ell_{m+1}$ for 
some $m \in \{1,\dots, k-1\}$.  By Part (1) of Proposition \ref{p:Shpairs},
$u^{*}=\tau_{\Delta} w\sigma^{-1}$, with $\tau_{\Delta}=(n,\dots, w(j),w(j+1),\ell_{m-1}, \dots,\ell_{1}).$
We compute: 
$u^{*}(\sigma(j))=\tau_{\Delta}(w(j))=w(j+1)$ and $u^{*}(\sigma(j+1))=\tau_{\Delta}(w(j+1))=\ell_{m-1}$, if $m\neq 1$.  If on the other hand $m=1$, then $u^{*}(\sigma(j+1))=n$.  Since $n$ is the minimal element in the total ordering $\prec$ on $\{1,\dots, n\}$, in either case we have $u^{*}(\sigma(j+1))\prec u^{*}(\sigma(j))$, whence $\alpha^{*}$ is complex unstable for $u^{*}$.  

Now suppose that $\alpha\in\Pi_{\fk}$ is a root of $\fk$ with 
$\alpha=\eps_{i}-\eps_{i+1}$ for $i=1,\dots, n-2$.  
Let $\Delta$ be as above but with $\ell_{m}=i$ and $\ell_{m+1}=i+1$ for some 
$m\in\{1,\dots, k-2\}$ (since $\ell_{m+1}=i+1 \le n-1$, $m+1 \le k-1$).  Now since 
$\Delta$ is a decreasing sequence for $w^{-1}$, we must have 
$w^{-1}(i+1)<w^{-1}(i)$ which is equivalent to the statement that 
$s_{\alpha}w<w$.  Thus, $\alpha$ is complex unstable for $w$.  
To see that $\alpha$ is complex unstable for $u^{*}$, we need to show that 
$(u^{*})^{-1}(i+1)\prec (u^{*})^{-1}(i)$.  By Proposition \ref{p:Shpairs}, we have 
$(u^{*})^{-1}=\sigma w^{-1} \tau_{\Delta}^{-1}$ with $\tau_{\Delta}^{-1}=(\ell_1,\dots, \ell_{m-1} , i, i+1, \ell_{m+2},\dots, n)$.   
We compute $(u^{*})^{-1}(i+1)=\sigma w^{-1}\tau_{\Delta}^{-1} (i+1)=\sigma (w^{-1}( \ell_{m+2}))$ and 
$(u^{*})^{-1}(i)=\sigma (w^{-1}(i+1))$.  Now since $\Delta$ is a decreasing sequence for $w^{-1}$, 
we know $w^{-1}(\ell_{m+2})<w^{-1}(i+1)$.  Further, since $\sigma=(n,n-1,\dots, 2,1)$, we have 
$\sigma(w^{-1}(i+1))=w^{-1}(i+1)-1$ and $\sigma(w^{-1}(\ell_{m+2}))=w^{-1}(\ell_{m+2})-1$ or $\sigma(w^{-1}(\ell_{m+2}))=n$.  In either case, 
$(u^{*})^{-1}(i+1)\prec (u^{*})^{-1}(i)$.  Thus, $\alpha$ is complex unstable for $u^{*}$ as well. 
\end{proof}

\begin{proof}[Proof of Proposition \ref{p:shnc}]
Let $\Sh(Q)=(w,u^{*})$ and let $\alpha=\eps_{j}-\eps_{j+1}\in\Pi_{\fg}$ be non-compact   
for $Q$. Then it follows from Part (2) of Theorem \ref{thm:intertwine} that $\alpha$ is complex stable 
for exactly one of $w$ and $u^{*}$ and unstable for the other.  Suppose first that 
$\alpha$ is stable for $w$ and unstable for $u^{*}$.  Then $\Sh(\ms*Q)=(ws_{\alpha}, u^{*})$ and 
Proposition \ref{p:Shpairs} implies that $\tau_{\Delta^{\prime}}=u^{*}\sigma(w s_{\alpha})^{-1}$ for
$\Delta^{\prime}\subset\{1,\dots, n\}$ a decreasing sequence for $(ws_{\alpha})^{-1}$.  We compute 
$\tau_{\Delta^{\prime}}=u^{*}\sigma w^{-1} s_{w(\alpha)}$, so that 
$\tau_{\Delta^{\prime}}=\tau_{\Delta}(w(j),\, w(j+1))$.  
Now $\tau_{\Delta^{\prime}}$ is a cycle, so at least one of $w(j)$ or $w(j+1)$ must be in 
the set $\Delta$.  Futher, Lemma \ref{l:bothindelta} implies that exactly one of $w(j)$ or $w(j+1)\in \Delta$.  
It follows that $\tau_{\Delta^{\prime}}$ is a cycle of length exactly $1$ more than the length of $\tau_{\Delta}$.  
Now suppose that $\alpha$ is unstable for $w$ but stable for $u^{*}$.  Part (2) of Theorem \ref{thm:intertwine} implies that $\Sh(\ms*Q)=(w, u^{*}s_{\alpha^{*}})$.  It then follows from Proposition \ref{p:Shpairs} that $\tau_{\Delta^{\prime}}=u^{*}s_{\alpha^{*}}\sigma w^{-1}$.  Using the fact that $s_{\alpha^{*}}=\sigma s_{\alpha} \sigma^{-1}$, the expression for $\tau_{\Delta^{\prime}}$ becomes 
$\tau_{\Delta^{\prime}}=u^{*}\sigma s_{\alpha}w^{-1}=u^{*}\sigma w^{-1} s_{w(\alpha)}=\tau_{\Delta} s_{w(\alpha)}$ as in the previous case.   The result of the Proposition follows in the same manner as above.  

Finally, suppose that $\alpha\in\Pi_{\fk}$ is non-compact   for $Q$.  Then by Part (2) of Theorem \ref{thm:intertwine} 
$\alpha$ is complex stable for exactly one of $w$ or $u^{*}$ and unstable for the other. 
For the case in which $\alpha$ is stable for $w$ but unstable for $u^{*}$, one computes using Theorem \ref{thm:intertwine} and Proposition \ref{p:Shpairs} that 
$\tau_{\Delta^{\prime}}=\tau_{\Delta}s_{\alpha}$.  In the other case, one computes that $\tau_{\Delta^{\prime}}=s_{\alpha}\tau_{\Delta}$.  The argument proceeds analogously to the case where $\alpha\in\Pi_{\fg}$ above using Proposition \ref{p:Shpairs} and Lemma \ref{l:bothindelta}.

\end{proof}

\begin{rem}\label{r:hatvectors}
Proposition \ref{p:shnc} and Remark \ref{r:whatisdelta} imply that if
$Q=B_{n-1}\cdot\F$ with $\F$ a flag in standard form and $\alpha\in\Pi_{\fk\oplus \fg}$ is non-compact   
for $Q$ with $Q^{\prime}=\ms*Q=B_{n-1}\cdot\F^{\prime}$ with $\F^{\prime}$ in standard form, then the 
flag $\F^{\prime}$ has exactly one more hat vector than $\F$. 

\end{rem}

\begin{thm}\label{t:dimformula}
Let $Q\in\Borbitspace$ with $\Sh(Q)=(w, u^{*})$ and let $\sigma\in\mathcal{S}_{n}$ be the $n$-cycle 
$\sigma=(n,n-1, \dots, 1)$.  Then 
\begin{equation}\label{eq:dimformula}
\dim Q=\frac{\ell(w)+\ell(u^{*})+|u^{*}\sigma w^{-1}|-n}{2} =\frac{\ell(w)+\ell(u^*)+|\Delta(w,u^*)|-n}{2},
\end{equation}
where $|u^{*}\sigma w^{-1}|$ denotes the order of the element $u^{*}\sigma w^{-1}$ in the group $\mathcal{S}_{n}$ and $\Delta(w,u^*)$ is from Remark \ref{r:whatisdelta}.
\end{thm}
\begin{proof}
Let $D(Q)$ be either of the two equivalent fractions in Equation (\ref{eq:dimformula}).   We prove that $\dim Q=D(Q)$ by induction on the dimension of $Q$.  First, suppose that $\dim Q=0$.  By Remark 6.9 of \cite{CE21II}, $Q=Q_{i}:=B_{n-1}\cdot \F_{i}$ where $\F_{i}$ is the flag in standard form 
$$
\F_{i}:=(e_{1}\subset\dots \subset e_{i-1}\subset\underbrace{e_{n}}_{i}\subset e_{i}\subset \dots \subset e_{n-1}) 
$$
for some $i=1,\dots, n$.  We claim that $D(Q)$ is also $0$. Since $\F_{i}$ does not contain any hat vectors, the first statement of Proposition \ref{p:orbits} implies that $\F_{i}=\widetilde{\F_{i}}=\F_{i}^{*}$.  Recalling Equation (\ref{eq:introgenerators}),  we compute $\Sh(Q_{i})=(w, u^{*})=(s_{n-1}\dots s_{i}, s_{1}^{*}\dots s_{i-1}^{*})$, whence $\ell(w)+\ell(u^{*})=n-1$.  Further, Remark \ref{r:whatisdelta} implies that $u^{*}\sigma w^{-1}=id$, so that $\ell(w)+\ell(u^{*})+|u^{*}\sigma w^{-1}|-n=n-1+1-n=0$, and $D(Q)=0.$

Now suppose that $\dim Q=k>0$ and assume that $\dim Q_1 = D(Q_1)$ whenever $\dim Q_1 \le k-1$.  By Theorem 6.5 of \cite{CE21II},  $Q=m(s_{\alpha})*Q^{\prime}$ for some $\alpha\in \Pi_{\fk\oplus \fg}$ with $\dim Q^{\prime}=k-1$.   Let 
$\Sh(Q^{\prime})=(y, v^{*})$.  Let $\Delta=\Delta(w,u^{*})$ and let $\Delta^{\prime}=\Delta(y,v^{*}).$
    By Definition \ref{d:roottype}, the root $\alpha$ is either complex stable or non-compact   for $Q$.  

 First, suppose that $\alpha$ is complex stable for $Q^{\prime}$.  If $\alpha \in \Pi_{\fg}$,  then by Part (1) of Theorem \ref{thm:intertwine}, it follows that $\alpha$ is complex stable for both $y$ and $v^{*}$ and thus $(w,u^*)=\Sh(Q)=(ys_{\alpha}, v^{*}s_{\alpha^{*}})$
by Equation (\ref{eq:diagonalmonoid}).   In addition, $\ell(w)=\ell(y)+1$ and $\ell(u^*)=\ell(v^*)+1.$  Consider the element 
 $\tau_{\Delta}=v^{*}s_{\alpha^{*}}\sigma (ys_{\alpha})^{-1}=v^{*}s_{\alpha^{*}}\sigma s_{\alpha} y^{-1}$.  Since
 $s_{\alpha^{*}}=\sigma s_{\alpha}\sigma^{-1}$, we see that $\tau_{\Delta}=v^{*}\sigma y^{-1}=\tau_{\Delta^{\prime}}$. 
Hence,
$$
D(Q)=\frac{\ell(w) + \ell(u^{*}) + |\Delta| - n}{2}=\frac{\ell(y)+\ell(v^{*})+2 + |\Delta^{\prime}|-n}{2}=D(Q^{\prime})+1=k-1+1=k
$$
by induction.  This establishes the result when $\alpha \in \Pi_{\fg}.$   Now suppose $\alpha \in \Pi_{\fk}.$   Then again by Theorem \ref{thm:intertwine} and Equation (\ref{eq:diagonalmonoid}),
 $(w,u^*)=\Sh(Q)=(s_{\alpha}y,s_{\alpha}v^{*}).$  Thus, $u^*\sigma w^{-1}=s_{\alpha}v^{*}\sigma y^{-1}s_{\alpha}$, which is conjugate to $v^*\sigma y^{-1}$, so that
$u^{*}\sigma w^{-1}$ and $v^*\sigma y^{-1}$ have the same order.  The rest of the argument proceeds as in the case $\alpha \in \Pi_{\fg}.$

Now suppose $\alpha\in\Pi_{\fk\oplus \fg}$ is non-compact   for $Q^{\prime}$.  It follows from Proposition \ref{p:shnc} that $|\Delta|=|\Delta^{\prime}|+1$.  Further, by Part (2) of Theorem \ref{thm:intertwine}, the root $\alpha$ is complex stable for one of $y$ and $v^{*}$ and is unstable for the other.  Therefore, $\ell(w)=\ell(y)+1$ and $\ell(u^{*})=\ell(v^{*})$ or $\ell(w)=\ell(y)$ and $\ell(u^{*})=\ell(v^{*})+1$.  It follows from the induction hypothesis that $\dim Q = D(Q)$ in this case as well.  Thus, Equation (\ref{eq:dimformula}) holds for any orbit of dimension $k$ and thus holds for all orbits in $\Borbitspace$ by induction. 



\end{proof}

Equation (\ref{eq:dimformula}) can be expressed more succinctly in terms of a variant of the Shareshian pair.  
\begin{dfn}\label{d:twistedSh}
Let $Q\in\Borbitspace$ and let $\Sh(Q)=(w,u^{*})$.  Then its {\it standardized Shareshian pair} is $\widetilde{Sh}(Q):=(w,u)$ where $u=:\sigma^{-1} u^{*}\sigma$. 
\end{dfn}
We define the Bruhat order on standardized Shareshian pairs to be the restriction of the product of standard Bruhat orders on $(W,S)\times (W,S)$.  
Clearly, for $Q,\, Q^{\prime}\in\Borbitspace$ we have $\Sh(Q^{\prime})\leq \Sh(Q)$ if and only if $\widetilde{Sh}(Q^{\prime})\leq \widetilde{Sh}(Q)$ 


Let $Q\in\Borbitspace$ with $\widetilde{Sh}(Q)=(w,u)$ and consider the element $uw^{-1}=\sigma^{-1}u^{*}\sigma w^{-1}=\sigma^{-1}\tau_{\Delta}w\sigma^{-1}\sigma w^{-1}=\sigma^{-1}\tau_{\Delta}$ by Part (1) of Proposition \ref{p:Shpairs}.  
Suppose $\Delta=\{j_{1}<j_{2}<\dots < j_{k}<n\}$.  We claim:
\begin{equation}\label{eq:specialelt}
uw^{-1}=\sigma^{-1}\tau_{\Delta}=s_{1}\dots s_{j_{1}-1}\hat{s}_{j_{1}}s_{j_{1}+1}\dots s_{j_{2}-1}\hat{s}_{j_{2}}s_{j_{2}+1}\dots \hat{s}_{j_{k}}\dots s_{n-1},
\end{equation}
where $\hat{s}_{j_{i}}$ indicates that $s_{j_{i}}$ has been omitted from the product $s_{1}s_{2}\dots s_{n-1}$.  To see this, decompose $\tau_{\Delta}$ as a product of transpositions 
$\tau_{\Delta}=(n,j_{k})(j_{k}, j_{k-1})\dots (j_{3}, j_{2})(j_{2},j_{1})$.  First, consider the product $\sigma^{-1}(n,j_{k})$.  Note that $\sigma^{-1}=s_{1}s_{2}\dots s_{n-1}$ and that 
 $(n,j_{k})=s_{\alpha_{j_{k}}+\dots +\alpha_{n-1}}$ so that 
$(n,j_{k})=s_{n-1}\dots s_{j_{k}+1} s_{j_{k}} s_{j_{k}+1}\dots s_{n-1}$.  Then  
\begin{eqnarray*}
&\sigma^{-1}(n, j_{k})=s_{1}\dots s_{n-1}(s_{n-1}s_{n-2}\dots s_{j_{k}+1} s_{j_{k}} s_{j_{k}+1}\dots s_{n-1})\\
&=s_{1}\dots s_{j_{k}-1}s_{j_{k}+1}\dots s_{n-1}= s_{1}\dots s_{j_{k}-1} \hat{s}_{j_{k}} s_{j_{k}+1}\dots s_{n-1}.
\end{eqnarray*}
Next, consider the product $\sigma^{-1}(n,j_{k})(j_{k},j_{k-1})$.  Similarly, $(j_{k},j_{k-1})$ decomposes as the product $s_{j_{k}-1}\dots s_{j_{k-1}+1} s_{j_{k-1}}s_{j_{k-1}+1}\dots s_{j_{k}-1}$.  Using the computation above, 
\begin{eqnarray*}&\sigma^{-1}(n,j_{k})(j_{k},j_{k-1})=s_{1}\dots s_{j_{k}-1}\hat{s}_{j_{k}}s_{j_{k}+1}\dots s_{n-1}(s_{j_{k}-1}\dots s_{j_{k-1}+1} s_{j_{k-1}}s_{j_{k-1}+1}\dots s_{j_{k}-1})\\
&=
s_{1}\dots s_{j_{k}-1}(s_{j_{k}-1}\dots s_{j_{k-1}}\dots s_{j_{k}-1})s_{j_{k}+1}s_{j_{k}+2}\dots s_{n-1}=s_{1}\dots \hat{s}_{j_{k-1}}\dots \hat{s}_{j_{k}}\dots s_{n-1}. 
\end{eqnarray*}
Continuing in this fashion, we obtain Equation (\ref{eq:specialelt}).

Given Equation (\ref{eq:specialelt}), we observe that $\ell(uw^{-1})=n-1-(|\Delta|-1)=n-|\Delta|$.  We can then rewrite Equation (\ref{eq:dimformula}) in 
terms of the standardized Shareshian pair for $Q$.  
\begin{cor}\label{c:twisteddim}
Let $Q\in\Borbitspace$ with $\widetilde{\Sh}(Q)=(w,u)$.  Then 
\begin{equation} \label{eq:normdimformula}
\dim Q=\frac{\ell(w)+\ell(u)-\ell(uw^{-1})}{2}.
\end{equation}
\end{cor}

\begin{exam}\label{ex:b2flag3}
We conclude the paper with the Bruhat graph for the $B_{2}$-orbits on the flag variety of $\fgl(3)$, which organizes the $B_2$-orbits according to their dimension and the closure relations between orbits.  We label orbits according to their standardized Shareshian pairs in order to indicate the utility of our results.    In the diagram below, the top row consists of zero dimensional orbits, and the dimension of the orbits increases by $1$ as we descend from row to row.   If $Q$ and $Q^{\prime}$ are two $B_{2}$-orbits, we indicate that $Q^{\prime} \subset \overline{Q}$ by exhibiting a sequence of downward lines from $Q^{\prime}$ to $Q.$     We also indicate the monoid actions.  A red line denotes a non-compact root and a blue line denotes a complex stable root.  A green line
indicates a closure relation that is not obtained from a monoid action.  A solid line indicates a right monoid action by a simple root of $\fg$, and a dashed line indicates a left monoid action by a simple root of $\fk$.  In the case, where either a simple root of $\fg$ or $\fk$ can be used the dashed line is omitted. 

For a standardized Shareshian pair $(w,u)$ and a root $\alpha\in\Pi_{\fg}$, we have 
$\ms*u \neq u$ if and only if $us_{\alpha}>u$, and for a root $\alpha\in \Pi_{\fk}$, $\ms*u\neq u$ if and only 
if $ s_{\sigma^{-1}(\alpha)}u> u$.   We let 
$\Pi_{\fg}=\{ \alpha, \beta\}$ with $\alpha=\eps_{1}-\eps_{2}$ and $\beta=\eps_{2}-\eps_{3}$ so that $\Pi_{\fk}=\{\alpha\}$, and we let $s=s_{\alpha}$ and $t=s_{\beta}$. 

\end{exam}

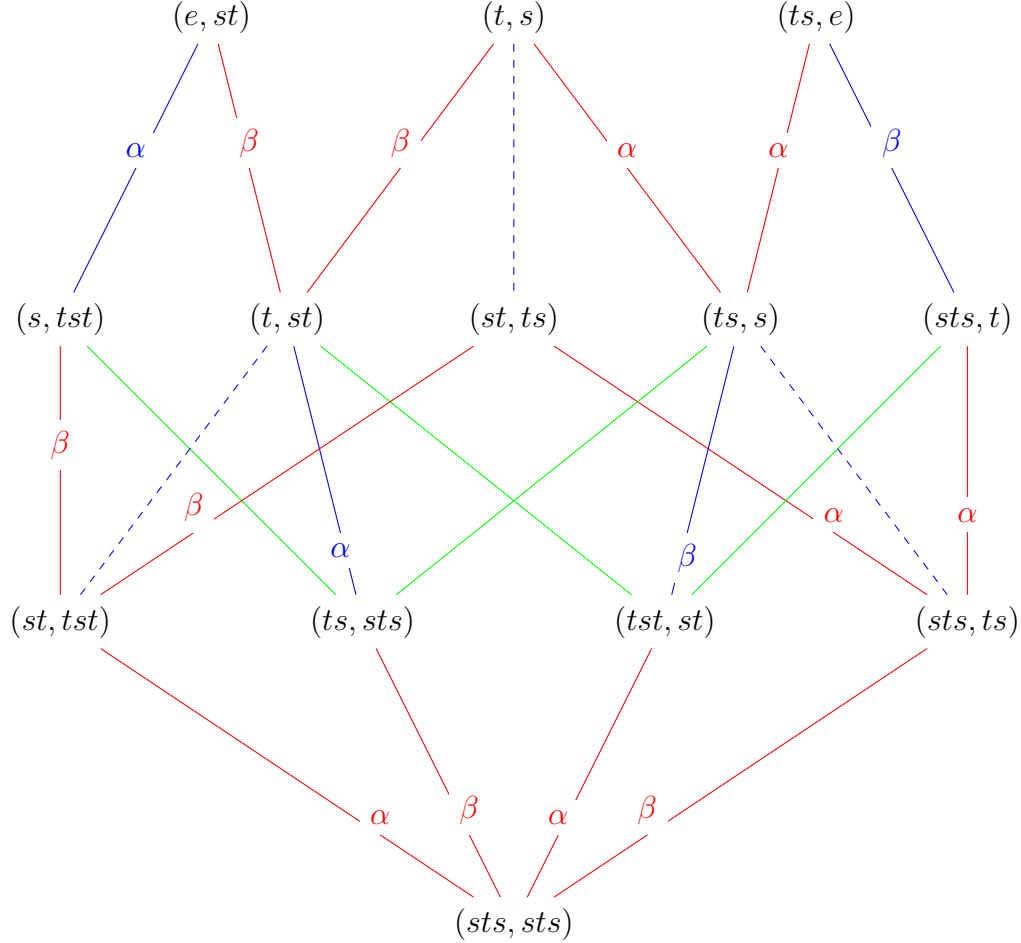
\begin{center}
\begin{tikzpicture}  
 [scale=2.0,auto=center,every node/.style={rectangle,fill=white!20}] 
\node (a1) at  (-1,5) {$(e, st)$};
\node (a2) at (1,5) {$(t,s)$};
\node (a3) at (3,5) {$(ts,e) $};
\node (a4) at (-2,3) {$(s,tst)$};
\node (a5) at (-0.5,3) {$(t,st)$};
\node (a6) at (1,3) {$(st,ts)$};
\node (a7) at (2.5,3) {$(ts,s)$};
\node (a8) at (4,3) {$(sts,t)$};
\node (a9) at (-2,1) {$(st, tst)$};
\node (a10) at (0,1) {$(ts,sts)$};
\node (a11) at (2,1){$(tst,st)$} ;
\node (a12) at (4,1) {$(sts,ts)$} ;
\node (a13) at (1,-1) {$(sts,sts)$};
\draw [blue] (a1) -- (a4) node[midway, above] {$\alpha$}; 
  \draw [red] (a1) -- (a5) node[midway, above] {$\beta$};  
 \draw [red] (a2) -- (a5) node[midway, above] {$\beta$};  
 \draw [dashed] [blue] (a2)--(a6);
 \draw [red] (a2)--(a7) node[midway, above] {$\alpha$};
  \draw [red] (a3)--(a7) node[midway, above] {$\alpha$};
   \draw [blue] (a3)--(a8) node[midway, above] {$\beta$};
    \draw [red] (a4)--(a9) node[midway, above] {$\beta$};
    \draw[green] (a4)--(a10); 
    \draw[dashed][blue] (a5)--(a9);
    \draw[blue] (a5) --(a10) node[near end, below]  {$\alpha$}; 
    \draw[green] (a5)--(a11); 
    \draw[red] (a6)--(a9) node[near end, above] {$\beta$}; 
    \draw[red](a6)--(a12) node[near end, above] {$\alpha$}; 
    \draw[green] (a7)--(a10);
    \draw[dashed][blue] (a7)--(a12);
    \draw [blue] (a7)--(a11) node[near end, below] {$\beta$};
    \draw [red] (a8)--(a12)node[near end, above] {$\alpha$};
    \draw[green] (a8)--(a11); 
    \draw[red](a9)--(a13) node[near end, above] {$\alpha$};
\draw[red](a10)--(a13) node[near end, above] {$\beta$};
\draw[red](a11)--(a13) node[near end, above] {$\alpha$};
\draw[red](a12)--(a13)node[near end, above] {$\beta$};
\end{tikzpicture}
\end{center}

\begin{rem}
A diagram showing the same information for $B_2$-orbits on $\B_{3}$ as the diagram above using standard forms instead of standardized Shareshian pairs appears in Example 7.1 in \cite{CE21II}.    The diagram above is much more transparent than the corresponding diagram in \cite{CE21II}.  Indeed, the closure relations in this diagram are transparent from the Bruhat order for $W={\mathcal S}_3$.  The monoid actions are given by left or right monoid action in $W$, except for the left monoid action (of a root of $\Pi_{\fk}$) on the second factor.  If $\alpha$ is a root of $\Pi_{\fk}$, and $Q$ is a $B_{n-1}$-orbit with standardized Shareshian pair $(w,u)$, then $\ms*Q$  has standardized Shareshian pair $(m(s_{\alpha})*_{L}w,m(s_{\sigma^{-1}(\alpha)})*_{L}u).$   Thus, both the closure relations and the monoid action are only slightly more difficult to compute than they are in the Weyl group.
\end{rem}

\section{Declarations} 

\subsection{Ethical Approval}
Not applicable.
\subsection{Competing Interests}
The authors have no relevant financial or non-financial interests to disclose. 
\subsection{Funding}
No funding, grants, or other support was received.
\subsection{Data Availability} 

Not applicable.

\bibliographystyle{amsalpha.bst}

\bibliography{bibliography-1}

\end{document}